\documentclass[10pt,reqno]{amsart}
\usepackage[a4paper, total={6in, 8in}]{geometry}
\usepackage{tikz-cd} 
\usepackage[english]{babel}
\usepackage[utf8x]{inputenc}
\usepackage[T1]{fontenc}
\usepackage{amsmath}
\usepackage{amstext}
\usepackage{amsfonts}
\usepackage{amssymb}
\usepackage{nicefrac}
\usepackage{amsthm}
\usepackage[normalem]{ulem}
\usepackage{mathabx}
\usepackage{soul}
\usepackage{mathtools}
\usepackage[shortlabels]{enumitem}
\usepackage{dsfont}
\usepackage{color}
\usepackage{graphicx}
\usepackage[colorinlistoftodos]{todonotes}
\usepackage[colorlinks=true, allcolors=blue]{hyperref}
\usepackage{float}
\usepackage[colorinlistoftodos]{todonotes}
\usepackage{theoremref}
\usepackage{enumitem}
\usepackage{dirtytalk}
\allowdisplaybreaks
\newtheorem{theorem}{Theorem}[section]
\newtheorem{lemma}[theorem]{Lemma}
\newtheorem{prop}[theorem]{Proposition}
\newtheorem{cor}[theorem]{Corollary}

\newtheorem{note}[theorem]{Note}

\newtheorem*{cor*}{Corollary}
\newtheorem*{thm*}{Theorem}
\newtheorem*{lem*}{Lemma}

\newtheorem*{prop*}{Proposition}
\theoremstyle{definition}
\newtheorem{definition}[theorem]{Definition}
\newtheorem{defn}[theorem]{Definition}
\newtheorem{example}[theorem]{Example}
\newtheorem*{defn*}{Definition}

\theoremstyle{remark}
\newtheorem{remark}[theorem]{Remark}

\newcommand{\Crit}{\operatorname{Crit}}
\newcommand{\Supp}{\operatorname{Supp}}

\newcommand{\Cl}{\operatorname{Cl}}
\newcommand{\Il}{\operatorname{\mathcal{I}}}

\newcommand{\cA}{\mathcal{A}}

\newcommand{\cB}{\mathcal{B}}
\newcommand{\cC}{\mathcal{C}}

\newcommand{\cI}{\mathcal{I}}

\newcommand{\cZ}{\mathcal{Z}}

\newcommand{\bZ}{{\mathbb{Z}}}

\newcommand{\fa}{\mathfrak{a}}
\newcommand{\fb}{\mathfrak{b}}
\newcommand{\fc}{\mathfrak{c}}
\newcommand{\fd}{\mathfrak{d}}

\newcommand{\fj}{\mathfrak{j}}
\newcommand{\fA}{\mathfrak{A}}
\newcommand{\fC}{\mathfrak{C}}
\newcommand{\fD}{\mathfrak{D}}
\newcommand{\fI}{\mathfrak{I}}
\newcommand{\fJ}{\mathfrak{J}}

\newcommand{\N}{{\mathbb{N}}}
\newcommand{\Z}{{\mathbb{Z}}}
\newcommand{\Q}{{\mathbb{Q}}}
\newcommand{\R}{{\mathbb{R}}}
\newcommand{\C}{{\mathbb{C}}}

\newcommand{\cone}{\mathtt{Cif}1}
\newcommand{\ctwo}{\mathtt{Cif}2}

\newcommand{\ctwoa}{\mathtt{Cif}2a}
\newcommand{\ctwob}{\mathtt{Cif}2b}
\newcommand{\ctwoc}{\mathtt{Cif}2c}
\newcommand{\ctwoR}[1]{\mathtt{Cif}2(#1)}

\newcommand{\set}[1]{\ddot{#1}}

\newcommand{\supp}{\operatorname{Supp}}

\newcommand{\id}{\operatorname{id}}

\newcommand{\E}{\mathbb{E}}
\newcommand{\trivgp}{\langle e \rangle}
\newcommand{\Span}{\operatorname{Span}}

\newcommand{\CE}{\E_{\mu}}

\newcommand{\norm}[1]{\left \lVert #1 \right \rVert}
\newcommand{\abs}[1]{\left \lvert #1 \right \rvert}
\renewcommand{\phi}{\varphi}

\newcommand{\CZ}{C^{*}_{r}(\Z)}

\author{Tattwamasi Amrutam}
\address{Ben Gurion University of the Negev.
	Department of Mathematics.
	Be'er Sheva, 8410501, Israel.
}
\email{tattwamasiamrutam@gmail.com}
\author{Eli Glasner}
\address{Tel-Aviv University.
	Department of Mathematics.
	Tel-Aviv, Israel.
}
\email{glasner@math.tau.ac.il}
\author{Yair Glasner} 
\address{Ben Gurion University of the Negev.
	Department of Mathematics.
	Be'er Sheva, 8410501, Israel.
}
\email{yairgl@math.bgu.ac.il}


\thanks{This research was supported by grants from the Israel Science Foundation:
ISF 1175/18 for the first author,
ISF 1194/19 for the second, and
ISF 2919/19 for the third. The first named author was also supported by research funding from the European Research Council
(ERC) under the European Union's Seventh Framework Program
(FP7-2007-2013) (Grant agreement No. 101078193)}

\date{January, 2024}

\begin{document}
\title[NON-ABELIAN FACTORS]{NON-ABELIAN FACTORS FOR ACTIONS OF $\mathbb{Z}$ AND OTHER NON-$C^*$-SIMPLE GROUPS.}

\begin{abstract}
Let $\Gamma$ be a countable group and $(X, \Gamma)$ a compact topological dynamical system.
We study the question of the existence of an intermediate $C^*$-subalgebra
$\mathcal{A}$
$$
 C^{*}_{r} (\Gamma) < \cA < C(X) \rtimes_r \Gamma,
$$
which is not of the form  $\mathcal{A} = C(Y) \rtimes_r \Gamma$,
corresponding to a factor map $(X,\Gamma) \to (Y,\Gamma)$.
Here $ C^{*}_{r} (\Gamma)$ is the 
reduced $C^*$-algebra of $\Gamma$ and 
$C(X) \rtimes_r \Gamma$ is
the reduced $C^*$-crossed-product of $(X,\Gamma)$.
Our main results are:
(1) For $\Gamma$ which is not $C^*$-simple, when $(X,\Gamma)$ admits
a $\Gamma$-invariant probability measure, then
such a sub-algebra always exists.
(2) For $\Gamma = \mathbb{Z}$ and $(X, \Gamma)$
an irrational rotation of the circle $X = \mathbb{R}/\mathbb{Z}$,
we give a full description of all these
non-crossed-product subalgebras.
\end{abstract}

\maketitle
\tableofcontents
\section{Introduction}
Let $\Gamma$ be a countable group and
$(X, \Gamma)$ a compact topological dynamical system. We form the reduced $C^*$-crossed product $C(X) \rtimes_{r} \Gamma$. 

A factor map $\pi : (X,\Gamma) \to (Y,\Gamma)$ of $\Gamma$-systems gives rise to an intermediate subalgebra $C^{*}_{r}(\Gamma) < C(Y) \rtimes_r \Gamma < C(X) \rtimes_r \Gamma$. 
Now a $C^*$-subalgebra $\mathcal{A}$ of the crossed product
which contains $C^{*}_{r}(\Gamma)$ is invariant under the 
group of automorphisms induced by $\Gamma$ on $C(X) \rtimes_r \Gamma$
hence, by analogy, we can think of intermediate subalgebras $\{\cA \ | \ C^{*}_{r} (\Gamma) < \cA < C(X) \rtimes_r \Gamma\},$ as 
non-commutative factors of the dynamical system $(X,\Gamma)$. In \cite{AK}, it was shown that every intermediate $C^*$-algebra $\mathcal{A}$ of the form $C_r^*(\Gamma)\subset\mathcal{A}\subset C(X)\rtimes_r\Gamma$ is simple, provided that $\Gamma$ is $C^*$-simple and $X$ is a minimal $\Gamma$-space. In a few cases (e.g.,~\cite{A19}), it is possible to show that all the intermediate algebras come from dynamical factors. Other results prove a relative statement. Suzuki \cite{Suz18}, for example, shows that when $(\Gamma, Y)$ is a free action, all the intermediate algebras of the form $C(Y) \rtimes_r \Gamma < \cA < C(X) \rtimes_{r} \Gamma$ come from intermediate dynamical factors $X \rightarrow Z \rightarrow Y$. In this work, we study the existence of non-commutative factors when the group $\Gamma$ is not $C^*$-simple. 

Our main results concern
the case when the group $\Gamma$ is the group of integers and the dynamical system is an irrational circle rotation. 

Let $\alpha \in 2\pi \left(\R \setminus \Q\right)$, $X = [-\pi,\pi]/(-\pi \sim \pi)$ and let $T:X \rightarrow X$ 
be the irrational rotation given by $T(x) = x+\alpha \pmod {2\pi}$. The dynamical system $(X,T)$ is minimal and uniquely ergodic with
Lebesgue measure $\mu$, given by
$$
\mu(f) = \frac{1}{2\pi}\int_{-\pi}^{\pi} f(x)\, dx \qquad \qquad f \in C(X),
$$
as the unique invariant probability measure on $X$. We construct the crossed product $\cC = \cC_{\alpha}= C(X) \rtimes_{r} \bZ$. Note that $C(X)\rtimes_r\bZ$ is the same as the full crossed product $C(X) \rtimes \bZ$ since $\mathbb{Z}$ is amenable (See \cite[Theorem~4.2.6]{BO:book}). Let $\pi_n:X \rightarrow X$ denote the $n$-fold covering maps $\pi_n(x)=nx \pmod{ 2\pi}$, and $\pi_n^*:C(X) \to C(X)$ the dual map, given by $f \mapsto f \circ \pi_n$. These
maps give rise to the $T$-invariant subalgebras $C^n(X) = \pi_n^*(C(X))$ and consequently to intermediate sub-crossed products $\cC^n = C^n(X) \rtimes_{r} \bZ$. By Suzuki's result \cite{Suz18} (also see \cite{ryo2021remark}) every intermediate algebra of the form $\cC^n < \cA < \cC$ is of the form $\cC^q$ for some $q$ dividing $n$. 

The situation in the absolute case turns out to be more complex. We show that there are unaccountably many intermediate subalgebras $\CZ < \cA < \cC$ which are not isomorphic to such crossed products. Still, the situation is nice enough, and our main theorem gives a complete classification of all the intermediate subalgebras of $\cC$. This classification, in turn, yields structural information about these subalgebras. To state our results, we first introduce some terminology. 

Let $\Il$ denote the collection of all closed two-sided ideals in $\CZ$. 
We denote by
$$\fA = \{\cA \ | \  \CZ < \cA < \cC\}, \qquad \fD = \{\fc: \Z \rightarrow \Il \ | \ \fc(0) = \CZ \},$$
the collections of all intermediate subalgebras, and all $\Il$ valued functions on $\Z$ such that $\fc(0)=\CZ$, respectively. We will refer to elements of $\fD$ as {\it{ideal functions}}. The 
{\it support} of an ideal function is defined as
\begin{equation} \label{eqn:support}
\Supp(\fc) = \{n \in \Z \ | \ \fc(n) \ne (0)\}.
\end{equation}
For an intermediate algebra $\cA \in \fA$ and for an ideal function $\fc \in \fD$ we set 
$$\fI_{\cA} = \{I \ | \ I \lhd \cA \} \qquad \fD_{\fc} = \{\fj: \Z \rightarrow \cI \ | \ \fj(n) \le \fc(n), \forall n \in \Z\}.$$ 
$\fI_{\cA}$ is the collection of all closed two-sided ideals in $\cA$. 
Note that such an ideal is automatically 
$^*$-invariant.
In the sequel, when we say `ideal,' we mean a closed, two-sided one.

We refer to the elements of $\fD_{\fc}$ as {\it $\fc$-ideal functions}. Note that $\fD_{\fc} \not \subset \fD$ because we do not require $\fj(0)=\CZ$ for $\fj \in \fD_{\fc}$. The {\it support} of a $\fc$-ideal function is defined by an equation analogous to Equation (\ref{eqn:support}). Since by definition $\fj(n) \leq \fc(n), \forall n \in \Z, $ we have $\Supp(\fj) \subset \Supp(\fc), \forall \fj \in \fD_{\fc}$.

All of these collections admit a natural lattice structure. The partial order on $\fA$ is given by inclusion. Similarly, on $\fD$  
we write $\fc_1 \preceq \fc_2$ when $\fc_1(n) \le \fc_2(n), \ \forall n \in\Z$. The lattice operations are given by: 

\vspace{2mm}
\begin{tabular}{lll}
$\cA_1 \wedge \cA_2 = \cA_1 \cap \cA_2$, & $\cA_1 \vee \cA_2 = \overline{\langle \cA_1,\cA_2 \rangle}$, 
& $\forall \cA_1,\cA_2 \in \fA$ \\
$(\fc_1 \wedge \fc_2)(n) = \fc_1(n) \cap \fc_2(n)$, 
& $(\fc_1 \overline{\vee} \fc_2)(n) = \overline{\fc_1(n) + \fc_2(n)}$, 
& $\forall \fc_1,\fc_2 \in \fD$ \\
\end{tabular}
\vspace{2mm}

\noindent Given $\cA \in \fA$ and $\fc \in \fD$ the lattice structure on $\fI_{\cA}$ and on $\fD_{\fc}$ is defined similarly. 

Via the Fourier transform $\CZ \cong C(S^1)$, and hence $\Il$ can be identified with 
the collection of closed subsets of the circle $\Cl(S^1)$. Under this identification, with every $\fc \in \fD$, we associate a {\it set function} $\set{\fc}: \Z \rightarrow 
\Cl(S^1)$, such that $\set{\fc}(0)=\emptyset$. 
By $\set{\fD}$, we denote the collection of all such functions. Similarly for 
$\fc \in \fD$ we define the collection $\set{\fD}_{\fc}$.

It is crucial to distinguish between the given dynamical system $(X,T)$ and the circle $S^1 = \hat{\Z}$ introduced in the above paragraph
(see Section \ref{sec:circles} below). Still, perhaps surprisingly, we will often use the irrational rotation given by $\tau(t)=t + \alpha \pmod {2\pi}$. Thus, $\tau$ represents the irrational rotation on $S^1$ by the same angle $\alpha$ appearing in the given irrational rotation $T:X \rightarrow X$. 
\begin{definition} \label{def:closed}
An ideal function $\fc \in \fD$ is called {\it{closed}} if the following properties are satisfied: 
\begin{itemize}
    \item [$\cone$] $\set{\fc}(-n) = \tau^{n}\set{\fc}(n), \ \forall n \in \Z$,
    \item [$\ctwo$] $\set{\fc}(m+n) \subset \tau^{-m} \set{\fc}(n) \cup \set{\fc}(m)$ for every $m,n \in \Supp(\fc)$.
\end{itemize}
By $\fC \subset \fD$, we denote the collection of closed ideal functions. 
\end{definition}

\noindent
We define functions 
$\Phi: \fD \rightarrow \fA$ and $\Psi: \fA \rightarrow \fD$ by:
$$\Phi(\fc) = \overline{\langle e^{inx} \eta \ | \ n \in \Z, \eta \in \fc(n)\rangle}, \quad \Psi(\cA)(n) = \{\eta \in \CZ \ | \ e^{inx} \eta \in \cA\}.$$
Note that here $e^{inx}$ is a function in $C(X)$. See Section \ref{Sec:prelim} for more details concerning these notations.

The meet and join operations that we defined above translate to a standard union and intersection: 
$$(\set{\fc}_1 \wedge \set{\fc}_2)(n) = \set{\fc}_1(n) \cup \set{\fc}_2(n), \quad 
(\set{\fc}_1 \overline{\vee} \set{\fc}_2)(n) = \set{\fc}_1(n) \cap \set{\fc}_2(n), \quad  
\forall \set{\fc}_1,\set{\fc}_2 \in \set{\fD}.$$ 
It is straightforward to verify that $\set{\fC}$, and hence $\fC$, are closed under the meet operation. On the other hand, the join of two closed ideal functions is not generally closed (see Example \ref{eg:join_not_closed}). We are, therefore, led to define the 
{\it closed join} operation:
$$\fc_1 \vee \fc_2 := \Psi \circ \Phi(\fc_1 
\overline\vee \fc_2).
$$

In Proposition \ref{prop:closed_join}, we will establish an explicit (albeit long) formula for the closed join. The meet and the closed join now give $\fC$ a structure of a lattice, which is not a sublattice of $\fD$. This is the lattice structure alluded to in item (\ref{itm:isom}) of our main theorem below.
 
\begin{theorem} (Main theorem) \label{thm:main}
Let $\Phi:\fD \rightarrow \fA$ and $\Psi:\fA \rightarrow \fD$ be as above, then:
\begin{enumerate}
\item \label{itm:GC} The maps $\Phi,\Psi$ form a (monotone) Galois connection between the two lattices $\fD,\fA$. Namely, for $\fc \in \fD, \cA \in \fA$ we have $\Phi(\fc) < \cA$ if and only if $\fc < \Psi(\cA)$. 
\item \label{itm:perfect} The connection is perfect on the algebra side. Namely $\cA = \Phi \circ \Psi(\cA)$ for every $\cA \in \fA$. 
\item \label{itm:closed_fc} $\fc = \Psi\circ \Phi(\fc)$ if and only if it is closed in the sense of Definition \ref{def:closed}.
\item \label{itm:isom} $\Phi\upharpoonright_{\fC}: \fC \rightarrow \fA$ is an isomorphism of lattices, with $\Psi$ as its inverse.
\end{enumerate}
\end{theorem}
It is worthwhile noting what is easy and what requires new ideas in the above theorem. Once all the definitions are in place, the fact that $\Phi,\Psi$ form a Galois connection is clear. We invite the readers to verify that before proceeding. From this, it follows, by standard properties of Galois connections, that $\Phi,\Psi$ form an isomorphism between the sublattices of closed elements on both sides. The deeper part lies in showing that this Galois connection does not degenerate in the sense that closed elements are ubiquitous on both sides. In particular, the connection is perfect in the sense that all elements are closed on the algebra side. Indeed, it is a priory, not clear at all why a given intermediate algebra $\cA \in \fA$ should contain any element of the form $e^{inx} \eta 
\ (n \not = 0, \ 0 \neq \eta \in \CZ)$. 
For any given algebra $\cA \in \fA$, with an associated ideal function $\fc = \Psi(\cA) \in \fC$, we can state a similar classification for all closed two-sided ideals of $\cA$.  
\begin{definition} \label{def:c_closed}
  Given $\fc \in \fC$ a $\fc$-ideal function $\fj \in \fD_{\fc}$ will be called {\it{closed}} it satisfies the following two conditions:
    \begin{itemize}
    \item [$\cone(\fc)$] $\set{\fj}(-n) = \tau^{n}\set{\fj}(n), \ \forall n \in \Z$,
    \item [$\ctwoR{\fc}$] $\set{\fj}(m+n) \subset \tau^{-m} \set{\fc}(n) \cup \set{\fj}(m)$ for every $n \in \Supp(\fc), m \in \Supp(\fj)$.
\end{itemize}
\end{definition}
For $\fc \in \fD$ and $\cA \in \fA$ we define: 
$\Phi_{\fc}: \fD_{\fc} \rightarrow \fI_{\Phi(\fc)}$ and $\Psi: \fI_{\cA} \rightarrow \fD_{\Psi(\cA)}$ by:
$$\Phi_{\fc}(\fj) = \overline{\langle e^{inx} \eta \ | \ n \in \Z, \eta \in \fj(n)\rangle}, \quad \Psi_{\cA}(I) (n) = \{\eta \in \CZ \ | \ e^{inx} \eta \in I\}.$$
\begin{theorem} \label{thm:main_ideal}
Fix $\fc \in \fD$ and set $\cA = \Phi(\fc)$. 
\begin{enumerate}
\item \label{itm:GCj} The maps $\Phi_{\fc},\Psi_{\cA}$ form a (monotone) Galois connection between the two lattices $\fD_{\fc},\fI_{\cA}$.
\item \label{itm:perfectj} The connection is perfect at $\fI_{\cA}$.
\item \label{itm:closed_fcj} $\fj = \Psi_{\cA} \circ \Phi_{\fc}(\fj)$ if and only if it is closed in the sense of Definition \ref{def:c_closed}. 
\item \label{itm:isomj} $\Phi_{\fc}\upharpoonright_{\fC}: \fC_{\fc} \rightarrow \fI_{\cA}$ is an isomorphism of lattices, with $\Psi_{\cA}$ as its inverse.
\end{enumerate}
\end{theorem}
As in the main theorem, the lattice structure on $\fC_{\fc}$ uses the closed join operation, which is defined as in Proposition \ref{prop:closed_join}.
\noindent
Let us turn to some applications of this structure theory. 
\begin{definition}
We say that an intermediate algebra $\cA \in \fA$ is {\it{residual}} if $\set{\Psi(\cA)}(n) \in \Cl(S^1)$ is nowhere dense for every $n \in \text{Supp}\left(\Psi(\mathcal{A})\right)$. An ideal $J \lhd \cA$ in a residual algebra 
is called a {\it{residual ideal}} if $\set{\Psi_{\cA}(J)}(n)$ is nowhere dense for every $n \in \text{Supp}\left(\set{\Psi_{\mathcal{A}}(J)}\right)$. We exclude the trivial algebra and the trivial ideal. These will not be considered residual even though they qualify formally because there is nothing nontrivial in their support.
\end{definition}
Here are some properties of these algebras following the classification (see Sub-section \ref{sub:residual}). 
\begin{theorem} \thlabel{cor:residual}
Let $\cA,\{\cA_i\}_{i \in I} \in \fA$ be residual algebras and $(0) \ne J \lhd \cA$ a closed two sided nontrivial ideal. Let $\fc,\{\fc_i\}_{i \in I},\fj$ be the associated ideal functions. Set 
$$\Omega = \bigcup_{m,n \in \Supp(\fc)} \tau^{-m} \set{\fc}(n)$$ Then 
\begin{enumerate}
    \item \label{itm:vee} 
    Every intermediate algebra containing a residual one is also residual.
    \item \label{itm:wedge} $\cA_1 \wedge \cA_2$ is residual. 
    \item \label{itm:supp_gp} $\Supp(\fc) = \{n \in \Z \ | \ \fc(n) \ne (0)  \}$ is a subgroup of $\Z$. 
    \item \label{itm:supp_ideal} $\Supp(\fc)= \Supp(\fj)$.
    \item \label{itm:residual_ideal} $\set{\fj}(k) \subset \Omega$ for every $k \in \mathbb{Z}$
    and in particular every non trivial ideal $J \lhd \cA$ is residual. 
    \item \label{itm:center_free} $\cA$ is center free. 
 \end{enumerate}
Property (\ref{itm:center_free}) holds more generally, even if we assume that $\Supp(\fc)(n)$ is nowhere dense for one $0 \ne n \in \Supp(\fc)$.
\end{theorem}

Residual algebras should be thought of as large intermediate subalgebras. At the other extreme stand the algebras highlighted in the following definition:
\begin{definition}
An intermediate algebra $\cA \in \cA$ with associated ideal function $\fc = \Psi(\cA)$ is called {\it{small}} if $\Supp(\fc)$ is finite. 
\end{definition}
Here are some properties of small intermediate subalgebras. 
\begin{prop} \thlabel{cor:small}
Let $\cA_1,\cA_2,\cA \in \fA$ be small subalgebras $\fc_1,\fc_2,\fc$ the corresponding ideal functions, then
\begin{enumerate}
    \item \label{itm:nonempty_int} $\fc(n)$ has a nonempty interior for every $0 \ne n$. In particular, a small subalgebra is never residual. 
    \item \label{itm:basic_dichotomy} A basic subalgebra $\cA_{q, P}$ (see Example \ref{eg:prim} and Subsection \ref{sec:basic}) is small if and only if it is not residual. 
    \item \label{itm:small_fd} $\cA$ is finite dimensional as a module over $\CZ$. 
\item \label{itm:small_center}
Many small algebras admit a non-trivial center.
\item  \label{itm:join_not_small}
There are small algebras $\cA_1,\cA_2 \in \fA$ such that
$\cA_1 \vee \cA_2$ is not small.
\end{enumerate}
\end{prop}

In Subsection \ref{sub:simple}, we use our machinery to study the simplicity of intermediate algebras. We give examples of some algebras that are simple and others that are not. A partial characterization for simplicity is given
in Propositions \ref{prop:not-simple} and \ref{prop:necessary}.

\begin{remark} \label{rem:symmetry}
{\bf The universal property of $\mathcal{C}$:}
Contrary to what the notation $\cC = C(X) \rtimes_{r} \Z$ 
seems to imply there is a symmetry between the two factors, $C(X)$ and $C^*(\mathbb{Z})$,
of the crossed product. This is realized by an involutive automorphism $\iota: \cC \rightarrow \cC$ with the property that $\iota C(X)=\CZ$. 
When $\mathcal{C}=\mathcal{C}_\alpha$ is viewed as a crossed product over the algebra $\CZ$ it is isomorphic to $\mathcal{C}_{-\alpha}$.

It follows that all our results concerning
the intermediate subalgebras $\CZ < \cA < \cC$, are also valid 
for the analogous intermediate subalgebras
$C(X) < \mathcal{A} < \mathcal{C}$,
when they are reformulated with respect 
to the rotation $R_{-\alpha}$. See Remark 
\ref{cor:symmetry} below.
\end{remark}

To sum up, we have: 
\begin{theorem}\label{not-crossedZ}
\thlabel{notcrossedprodZ}
There are uncountably many intermediate subalgebras $\CZ < \cA < \cC$ which are not of the form $\mathcal{A}= C(Y) \rtimes_r \Z$, with
$X \to Y$ a dynamical factor of the system $(X,R_\alpha)$.
They are indexed by the system of ideal functions
$\mathfrak{C}$ as described in Theorem \ref{thm:main}.
 \end{theorem}

The fact that the crossed product $C(X) \rtimes_r \mathbb{Z}$
admits intermediate sub-algebras, which are not themselves crossed products,
are not at all special for $\mathbb{Z}$ dynamical systems. Let $(X,\Gamma)$ be a dynamical system that admits an invariant probability measure of full support. For such a dynamical system $(X,\Gamma)$, in Section \ref{sec:not-simple}, we establish that to every non-trivial ideal $I < C^*_r(\Gamma)$, corresponds
an intermediate subalgebra $\mathcal{B}_I$,
$$
C^*_r(\Gamma) \subset \mathcal{B}_I \subset C(X) \rtimes_r \Gamma,
$$
such that $\mathcal{B}_I\cap C(X)=\mathbb{C}$.
Such a non-trivial ideal exists if and only if $\Gamma$ is not $C^*$-simple. Moreover, this construction is also valid for general $\Gamma$-$C^*$-algebras $\mathcal{A}$ as long as $\mathcal{A}$ admits a faithful $\Gamma$-invariant state. However, to conclude that $\mathcal{B}_I$ is not a crossed product $C^*$-subalgebra, we need additional assumptions on the group $\Gamma$, or the $C^*$-algebra $\cA$.
More precisely,
we have:

\begin{theorem} \label{not-crossed}
\thlabel{notcrossedprod}
Let $\Gamma$ be a discrete group that is not $C^*$-simple and $\mathcal{A} \ne \C$ a $\Gamma$-$C^*$-algebra admitting a faithful $\Gamma$-invariant state. Assume that at least one of the following conditions holds:
\begin{itemize}
\item $\cA = C(X)$ with $\abs{X} \ge 3$. 
\item $\Gamma_f \ne \nicefrac{\Z}{2\Z}$, where $\Gamma_f$ denotes the FC-center of $\Gamma$.
\end{itemize}
Then, there is an intermediate $C^*$-algebra  $C_r^*(\Gamma)\subset\mathcal{D}\subset \mathcal{A}\rtimes_r\Gamma$ which is not of the form 
$\mathcal{D}= \mathcal{B} \rtimes_r \Gamma$ for any $\Gamma$-invariant subalgebra
$\mathcal{B} \subset \mathcal{A}$.
\end{theorem}
Note that $\Gamma_f = \trivgp$ exactly when $\Gamma$ is an i.c.c. group. In this case, we show (see~\thref{i.c.c-intermediate}) that every non-trivial ideal $I< C_r^*(\Gamma)$ corresponds to an intermediate sub-algebra $\mathcal{B}_I$ which is not a crossed product $C^*$-algebra.

Contrary to the spirit of the above theorem, Example~\ref{ex:necessaryassumption} shows that for the action of the (non $C^*$-simple) group $\Gamma=\mathbb{Z}/2\mathbb{Z}$ on $X = \{p_1,p_2\}$ there is no intermediate subalgebra $\cB$ of the form $C_r^*(\mathbb{Z}/2\mathbb{Z})\subset\cB\subset C(X)\rtimes_r\mathbb{Z}/2\mathbb{Z}$ at all. This example is responsible for the assumption in the second bulleted alternative in theorem ~\ref{notcrossedprod}.

Using similar methods, we also show that, for a $\Gamma$-$C^*$-algebra $\cA$, every non-trivial two-sided closed $\Gamma$-invariant ideal $I\le\cA$ corresponds to an intermediate object $\cA\subset\cB\subset\cA\rtimes_r\Gamma$ such that $\cB$ is not a crossed product of the form $\cA\rtimes_r\Lambda$ for any normal subgroup $\Lambda\triangleleft\Gamma$ (see~\thref{notcomingfromasubgroup}).

\vspace{.3cm}
We attempt to generalize our classification in the subsequent two works in progress.
First, to the general setting of $C(X) \rtimes_r\Gamma$ with $X$ is a compact Abelian group and $\Gamma$ a countable dense subgroup of $X$. In the second, we treat
analogous questions regarding intermediate von Neumann algebras
$\mathcal{N}$ of the form $L(\Gamma)\subset\mathcal{N}\subset \mathcal{M}\rtimes\Gamma$,
 where $\mathcal{M}$ is a von Neumann algebra. 
In particular, we study the case where $\mathcal{M}$ is the
commutative $L^\infty(X,\mu)$ equipped with the automorphism corresponding to the irrational rotation $R_\alpha$. 

\vspace{.3cm}

When we showed a draft of our work to Ilan Hirshberg, he pointed out that our classification results 
are closely related and, in fact, partly overlap several existing works in the literature.
More specifically, the works of Ruy Exel (see \cite{exel1994circle}), where the author shows the existence and investigates some of the properties of our ideal functions in a much more general setup.
Then, his extensive book on ``Fell Bundles'' \cite{exel2017partial},
and in particular, the results in Chapter 23 of this book.
The second source of examples arises as a result of
the dual nature of $\mathcal{C}$.
Our construction of ideal functions turns out to be
a generalization (within $\mathcal{C}$)
of the well-known Putnam's 
``$Y$-orbit breaking'' subalgebras, as described, for example, in 
Phillips' contribution to the book \cite{Phill} (see Subsection \ref{Putnam}).
Nonetheless, our approach and methods seem new, and so does our complete classification of intermediate subalgebras.\\

\noindent
\textbf{Acknowledgements:} We thank Ilan Hirshberg, Hanfeng Li, Mehrdad Kalantar, Sven Raum, and Yongle Jiang for their helpful remarks, discussions, and corrections. 
Finally, we thank the anonymous reviewer for carefully reading our manuscript and for many 
insightful comments and suggestions. In particular, we thank the referee for pointing out two serious mistakes:
(i) in  the 
formula for the closed ideal function of the join of two intermediate algebras, and (ii)
in our original proof of Proposition \ref{prop:Q} which is now replaced by a proof kindly suggested by the referee.
\vspace{.3cm}

\section{Some notations and preliminaries}\label{Sec:prelim}
\subsection{The crossed product \texorpdfstring{$C(X) \rtimes_r \mathbb{Z}$}{}} \label{sec:cp}

Let $(X,\mu,T)$ be a metric minimal free uniquely ergodic cascade (i.e. a $\mathbb{Z}$-action, $n \mapsto T^n$, where $T$ is a homeomorphism of $X$). Let $\mathcal{C} = C(X) \rtimes_r \mathbb{Z}$ be the corresponding
reduced crossed product $C^*$-algebra and, as usual, we denote by $C_r^*(\mathbb{Z})$
the $C^*$-algebra of $\mathbb{Z}$.
Via the Fourier transform $\mathcal{F} : \ell^2(\mathbb{Z}) \to L^2(S^1,\frac{dt}{2\pi})$, 
$(a_n) \mapsto \sum_{n \in \mathbb{Z}} a_n e^{int}$ (with $S^1 = \mathbb{R}/2\pi\mathbb{Z}$),
we can view $C_r^*(\mathbb{Z})$ as the commutative
algebra $C(S^1)$, where for $\phi \in C(S^1)$ the corresponding 
operator in $C^*(\mathbb{Z})$ is given by multiplication 
$v \mapsto \phi v,  \ v \in L^2(S^1)$.
Under this correspondence, we identify $\lambda_1$ 
(which corresponds to the shift operator on $\ell^2(\mathbb{Z})$) 
with the function $e^{it} \in C(S^1)$, and 
$\lambda_n$ with $e^{nit}$,  for every $n \in \Z$. We will denote the opposite isomorphism
$C(S^1) \rightarrow \CZ$, by $\phi \mapsto \hat{\phi}$. For trigonometric polynomials (and other nice functions), this is explicitly given by regular Fourier series $\hat{\phi}:=\sum_{n \in \Z} \hat{\phi}(n) \lambda_n$ with $\hat{\phi}(n) = \frac{1}{2\pi} \int_{S^1} \phi(t) e^{-int}dt$, and then extended by continuity to the whole of $C(S^1)$. 

We then identify $\mathcal{C}$ with the $C^*$-subalgebra of $\mathbb{B}(L^2(X,\mu))$
generated by the (unitary) Koopman operator $U_T$, which 
we also denote as $\lambda_1$, and the commutative $C^*$-algebra
of multiplication operators $M_f : v \mapsto fv, \ v \in L^2(X,\mu), \ f \in C(X)$.
We then let $\lambda_n = U_T^n$, so in particular $\lambda_0 = I$
and $\lambda_{-1} = U_T^{-1} = \lambda_1^*$.
A general element $a$ of the group ring over $C(X)$ (which is dense in $\mathcal{C}$)
has the form
$$
a = \sum_{n=-N}^N M_{f_n} U_T^n
$$
which, when no confusion can arise, we also write as
$$
a = \sum_{n=-N}^N f_n \lambda_n.
$$
With these notations the group ring that generates $C_r^*(\mathbb{Z})$ 
consists of elements of the form
$$
a = \sum_{n=-N}^N a_n \lambda_n,
$$
with $a_n \in \C$.

Recall that the $T$-action on $C(X)$ is implemented by conjugation
by $\lambda_1$:
$$
\lambda_1 M_f \lambda_1^* = M_{f \circ T},
$$
or simply
$$
\lambda_1 f \lambda_1^* = f \circ T.
$$

The reduced crossed product $C(X)\rtimes_r\mathbb{Z}$ comes equipped with a $\mathbb{Z}$-equivariant canonical conditional expectation $\mathbb{E}:C(X)\rtimes_r\mathbb{Z}\to C(X)$ defined on the algebraic crossed product by 
$$
\mathbb{E}(a) = \mathbb{E}\left(\sum_{n=-N}^N f_n \lambda_n\right)= f_0.
$$
It follows from \cite[Proposition~4.1.9]{BO:book} that $\mathbb{E}$ extends to a faithful conditional expectation from $C(X)\rtimes_r\Gamma$ onto $C(X)$.

For a general element $a \in \mathcal{C}=
C(X)\rtimes_r\mathbb{Z}$, its {\it Fourier coefficient
of order $n$} is the function
$$
\hat{a}(n) = \mathbb{E}(a\lambda^*_n) \in C(X).
$$

\subsection{Two circles} \label{sec:circles}

From here until the end of Section
\ref{sec:structure}, unless we say otherwise,
we take $(X,\mu,T)$ to be the circle $X$,
equipped with normalized Lebesgue measure $\mu$,
and with $T = R_\alpha$ an irrational rotation.
Thus the set of characters $\{e^{inx} | \ n \in \mathbb{Z}\}$ forms a basis for $L^2(X,\mu)$
and their linear combinations form a dense
$*$-subalgebra of $C(X)$.

When $(X, T)$ is an irrational rotation on the circle, it is important not to confuse the two circles. We will always refer to the Gelfand dual of $\CZ$ as $S^1$ and denote functions $\phi(t),\psi(t) \in C(S^1)$ by the lowercase Greek letters, with a variable $t$. The phase space of the dynamical system will always be denoted by $X$, and functions on $X$ will be denoted by Roman letters, $f(x),g(x) \in C(X)$ with a variable $x$. The action of $\Z$ on $X$ is by the rotation $T:X \rightarrow X$, $Tx=x+\alpha \pmod{2\pi}$. We will also consider the rotation, by the same angle $\alpha$, on the other circle and denote it by $\tau: S^1 \rightarrow S^1$, $\tau t = t + \alpha \pmod{2\pi}$. 

The rotation $\tau$ yields an action on $C(S^1)$ by $\phi \mapsto \phi \circ \tau$. Via the standard identification of $\CZ \cong C(S^1)$ we obtain an action $\tau: \CZ \rightarrow \CZ$, which is explicitly given in terms of Fourier coefficients by $\widehat{\tau \phi}(n) = \hat{\phi}(n) e^{in\alpha}$.
We denote the normalized Lebesgue measure
on $S^1$ by $\nu$. It is the unique $\tau$-invariant
Borel probability measure on $S^1$.
Note that $\mathbb{E} : \mathcal{C} \to C(X)$, our canonical conditional expectation, is given by
$$
\mathbb{E}\left(\sum_{-N}^N f_n \lambda_n\right) =
\mathbb{E}\left(\sum_{-N}^N f_n \widehat{e^{int}}\right) =
\sum_{-N}^N f_n \int_{S^1}e^{int}\; d\nu(t) = f_0.
$$
The $\tau$-action naturally extends to an action on the space of closed ideals $\Il$ in $\CZ$, identified with the closed subsets of the circle $\Cl(S^1)$. By abuse of notation, we will denote all of these by the same letter $\tau$. 
The following discussion yields the involutive automorphism alluded to in Remark \ref{rem:symmetry}
\begin{remark}
\label{cor:symmetry}

There exists an involutive automorphism $\iota:\cC \rightarrow \cC$ with $\iota C(X) = \CZ$. 
As is shown e.g. in \cite{Davidson} the $C^*$-algebra 
$\mathcal{C} = \mathcal{C}_\alpha, \ \alpha \in [0,1)$ irrational, 
is isomorphic to the universal $C^*$-algebra $C^*(U,V)$, with $U, V$ two unitaries satisfying the relation
\begin{equation} \label{eqn:universal}
UV = e^{2\pi i\alpha} VU. 
\end{equation}
It is also shown that $K_0(\mathcal{C}_\alpha) =
\mathbb{Z} + \mathbb{Z}\alpha$
and it follows that two such algebras
$\mathcal{C}_\alpha$  and $\mathcal{C}_\beta$
are isomorphic if and only if $\beta = \pm \alpha$. In our setup we have $U = M_{e^{ix}} \in C(X)$ and $V = \lambda_1 \in C^*(\mathbb{Z})\cong C(S^1)$.
If we take $U= \lambda_1 =  M_{e^{it}} \in C(S^1)$ and 
$V = M_{e^{ix}} \in C(X)$ we obtain exactly Equation (\ref{eqn:universal})
by Lemma \ref{lem:alg_opp} below
(with $\tau^{-1}$ replacing $T$).
\end{remark} 
\subsection{Generalized Fourier coefficients}\label{generalizedF} 
In addition to the standard Fourier coefficients of functions $\phi \in C(S^1)$ defined above, we use $\CZ$-valued Fourier expansion on $X$. 

Using $\mu$, the Lebesgue probability measure on $X$, we define a $\mathbb{Z}$-equivariant conditional expectation $\CE:\cC \to \CZ$, by setting $\CE(f\lambda_n) = \mu(f)\lambda_n$ and extending to the whole of $\cC$ using linearity and continuity (see \cite[Exercise~4.1.4]{BO:book}). Be careful not to confuse this with the standard conditional expectation $\E: \cC \to C(X)$.

\begin{defn} \label{def:GFC}
Let $a \in \cC$ we define its $\CZ$ valued Fourier coefficients by the following formula 
$$\overbracket{a}(n) = \CE \left( e^{-inx} a \right).$$
\end{defn}
For finite sums of the form 
\begin{equation} \label{eqn:Fourier}
a = \sum_{k=-N}^{N}e^{inx}\hat{\phi}_n, \quad \phi_n \in C(S^1).
\end{equation}
we just have $\overbracket{a}(n) = \hat{\phi}_n$. We will refer to such finite sums as {\it{generalized trigonometric polynomials}}. For a function $f \in C(X)$,  $$\overbracket{f}(n) = \int_{X} e^{-inx}f(x)\frac{dx}{2\pi},$$ is the standard $n^{th}$ Fourier coefficient. Note that our notation clearly distinguishes this from Fourier coefficients $\hat{\phi}(n)$ for a function $\phi \in C(S^1)$. 

Given $a \in \cC$ we will denote by 
$$
Q_n(a) = \sum_{k=-n}^{n} e^{ikx} \overbracket{a}(k)
=  \sum_{k=-n}^{n} e^{ikx} \hat{\phi}_k.
$$
Such {\it{generalized Fourier polynomial approximations}} play a central role in our arguments. Note that these differ from the more standard way of approximating elements of $\cC$ by expressions of $\sum f_k \lambda_k, \ f_k \in C(X)$.

Heuristically, one can think of general elements of $\cC$ as infinite Fourier expansions of the form $\sum_n e^{inx} \overbracket{a}(n)$. Of course, such infinite series often fail to converge in the norm. We will overcome this by approximating elements of the algebra using the following (also see \cite[Proposition~3.3]{svensson2009commutant}):

\begin{theorem}\label{thm:Fejer}(The Fej\'{e}r approximation theorem) 
The following two sequences, of generalized Ces\`{a}ro means of an element $a \in \mathcal{C}$
\begin{eqnarray*}
\hat \sigma_n(a) & = & 
\sum_{j= -n}^n \left(1 - \frac{\abs{j}}{n+1}\right) \hat a(j) \lambda_j \\
{\overbracket{\sigma}}_n(a) & = &
\sum_{j= -n}^n \left(1 - \frac{\abs{j}}{n+1}\right) e^{ijx} 
\overbracket{a}(j).
\end{eqnarray*}
Both converge to $a$ in the norm. 
\end{theorem}

\begin{note}\label{Fej-note}
For any $a \in \cC$ and fixed $q \in \N$ we have that  $$\lim_{n \rightarrow \infty} \norm{Q_q({\overbracket{\sigma}}_n(a)) - Q_q(a)} = 0.$$
\end{note}
\subsection{A basic lemma}

\begin{lemma} \label{lem:alg_opp}
Let $m,n \in \Z$ and $\hat{\phi},\hat{\psi} \in \CZ$ then:
\begin{enumerate}
    \item  \label{itm:c1} $\hat{\phi} e^{inx} = e^{inx} \widehat{\phi \circ \tau^{n}}, \quad \forall n\in\Z, \hat{\phi} \in \CZ$,
    \item \label{itm:c2} $\left(e^{inx} \hat{\phi}\right)^* = e^{-inx} \widehat{\overline{\phi} \circ \tau^{-n}}$,
    \item \label{itm:c3} $e^{inx}\hat{\phi} e^{imx}\hat{\psi} = e^{i(n+m)x} 
    \left[\widehat{(\phi \circ \tau^m)\psi}\right]$.
\end{enumerate}
\end{lemma}
\begin{proof}
The first property directly implies the other two. Also, it is enough to prove the first property when $\phi$ is a finite sum of the form $\hat{\phi} = \sum_{k=-K}^{K} a_k \lambda_k$ because such functions are uniformly dense in $C(S^1)$. In this case, we obtain a direct computation:
\begin{eqnarray*}
\sum_{k} a_k \lambda_k e^{inx} & = & \sum_k a_k \lambda_k e^{inx} \lambda_{-k} \lambda_k 
= \sum_{k} a_k e^{in(x+k\alpha)} \lambda_k  \\
& = & e^{inx} \sum_{k} a_k e^{ikn\alpha} \lambda_k = 
e^{inx} \widehat{\phi \circ \tau^{n}}
\end{eqnarray*}
Where the last equality is a standard property of Fourier coefficients.
\end{proof}

\subsection{Ideal functions} \label{def:ideal}

Given an intermediate algebra $\cA \in \fA$, with any function $f \in C(X)$ we associate a closed ideal of $\CZ$ 
\begin{eqnarray*}
I_{\cA}(f) & = & \{\eta \in \CZ \ | \ f \eta \in \cA \} \in \Il.
\end{eqnarray*}
Recall that $\cI$ was defined in the introduction as the set of all closed ideals in $\CZ$. 
\begin{defn} \label{def:Phi_inv}
Given an intermediate algebra $\CZ < \cA < \cC$ we define {\it{the ideal function}} $\Psi(\cA) = \fc_{\cA}: \Z \rightarrow \Il$ {\it{associated with $\cA$}} by
\begin{equation*}
\fc_{\cA}(n) = I_{\cA}(e^{inx}).
\end{equation*}
Note that $\fc_{\cA}(0) = \CZ$ since by assumption $\CZ < \cA$, so that $\fc_{\cA} \in \fD$. 
The {\it support} of $\fc_{\cA}$ is the set 
$$
\Supp(\fc_{\cA}) = \{n \in \mathbb{Z} \ | \ \fc_{\cA}(n) \ne (0)\}
= \{n \in \mathbb{Z} \ |\  \set{\fc}(n) \not = S^1\}.
$$ 
and we say that $\fc_{\cA}$ is trivial if $\Supp(\fc_{\cA}) = \{0\}$.
\end{defn}

It is unclear that $\fc_{\cA}$ should be nontrivial for a given $\CZ \ne \cA \in \fA$. Still, this turns out to be true, and our main Theorem \ref{thm:main} completely classifies intermediate algebras in terms of their ideal functions. With this in mind, we study the basic properties of an ideal function that comes from an algebra. 
The basic Lemma \ref{lem:alg_opp} is our main tool.
\begin{prop} \label{cor:O_props}
Let $\cA \in \fA$ and $\fc = \Psi(\cA)$. The following properties are satisfied: 
\begin{itemize}
    \item [$\cone$] $\set{\fc}(-n) = \tau^{n}\set{\fc}(n), \ \forall n \in \Z$,
    \item [$\ctwo$] $\set{\fc}(m+n) \subset \tau^{-m} \set{\fc}(n) \cup \set{\fc}(m)$ for every $m,n \in \Supp(\fc)$.
\end{itemize}
\end{prop}

\begin{proof} 
Let $\phi,\psi \in C(S^1)$ be functions with $\cZ(\phi) := \{t \in S^1 \ | \ \phi(t)=0\} = \set{\fc}(n)$ and $\cZ(\psi) = \set{\fc}(m)$. 
Namely $\hat{\phi},\hat{\psi}$ are generators of $\fc(n),\fc(m)$ respectively. Now $ \cZ((\phi \circ \tau^m)\psi) = \tau^{-m} \set{\fc}(n) \cup \set{\fc}(m)$ and $\ctwo$ follows directly from equation (\ref{itm:c3}) of Lemma \ref{lem:alg_opp}. Similarly applying Equation (\ref{itm:c2}) of the same lemma, to both $\pm n$, yields property $\cone$.
\end{proof}

\begin{defn}
\thlabel{closedidealfunction}
An ideal function $\fc \in \fD$ is called {\it{closed}} if it satisfies properties $\cone,\ctwo$. We denote the collection of closed ideal functions by $\fC \subset \fD$. 
As noted in the introduction, $\fC \subset \fD$ is not a sublattice, as it does not respect the join operation. We will elaborate on this later.
\end{defn}

\begin{note}
Condition $\ctwo$ can be rephrased more symmetrically:
$$\set{\fc}(m+n) \subset \left(\tau^{-m} \set{\fc}(n) \cup \set{\fc}(m)\right) \cap \left(\tau^{-n} \set{\fc}(m) \cup \set{\fc}(n)\right), \ \forall m,n \in \Supp(\fc)$$
\end{note}

\begin{note} \label{note:pos_gen}
Condition $\cone$ implies that every closed ideal function is completely determined by its values on $\N$. Thus, every function $\fc: \N \rightarrow \Il$ subject to the following three conditions admits a unique extension to a closed ideal function defined on all of $\Z$.
\begin{itemize}
\item [$\ctwoa$] $\set{\fc}(m+n) \subset \tau^{-m} \set{\fc}(n) \cup \set{\fc}(m), \quad \forall m,n \in \N \cap \Supp(\fc)$
\item [$\ctwob$] $\set{\fc}(m-n) \subset \tau^{n-m} \set{\fc}(n) \cup \set{\fc}(m), \quad \forall 1 \le n < m, \quad m,n \in \Supp(\fc)$
\item [$\ctwoc$] $\set{\fc}(-n+m) \subset \tau^{n}(\set{\fc}(n) \cup \set{\fc}(m)), \quad \forall 1 \le n < m, \quad m,n \in \Supp(\fc)$
\end{itemize}
\begin{proof}
We extend $\set{\fc}$ to $\Z$ by setting $\set{\fc}(0)=\emptyset, \set{\fc}(-n) = \tau^n \set{\fc}(n), \forall n \in \N$. With these definitions, condition $\ctwo$ holds when $m=n$. Hence, without loss of generality, we assume that $m > n$. This leaves us with the task of verifying $\ctwo$ in six different cases, three of which are the given equations. Let us verify that, assuming $\ctwoa$, $\ctwob$, $\ctwoc$ the other three hold as well: 
\begin{eqnarray*}
\set{\fc}(n-m) & = & \tau^{m-n}\set{\fc}(m-n) \subset \tau^{m-n} \left(\tau^{-m} \tau^n \set{\fc}(n) \cup \set{\fc}(m) \right) = \set{\fc}(n) \cup \tau^{-n} \set{\fc}(-m) \\
\set{\fc}(-m+n) & = & \tau^{m-n}\set{\fc}(-n+m) \subset \tau^{m-n} \tau^{n} \left(\set{\fc}(n) \cup \set{\fc}(m) \right) = \tau^{m} \left(\set{\fc}(n) \cup \set{\fc}(m) \right)  \\
\set{\fc}(-n-m) & = & \tau^{m+n} \set{\fc}(m+n) \subset \tau^{m+n} \left(\tau^{-m} \set{\fc}(n) \cup \set{\fc}(m) \right) = \set{\fc}(-n) \cup \tau^{n}\set{\fc}(-m)
\end{eqnarray*}
Which is exactly what we had to show. 
\end{proof}
\end{note}
\begin{example} \label{eg:prim}
Given $q \in \N$ and a closed subset $P \in \Cl(S^1)$. We define a closed ideal function $\fb_{q,P} \in \fC$ by
$$
\set{\fb}_{q,P}(m) = \left\{ 
   \begin{array}{ll} 
   \tau^{q(1-n)}P \cup \ldots \cup \tau^{-q}P \cup P & m = nq, \ n > 0 \\
   \emptyset & m = 0 \\
   \tau^{q} P \cup \tau^{2q} P \cup \ldots \cup \tau^{nq}P & m=-nq, \ n > 0   \\
   S^1 & q \not{\mid} m  
   \end{array} \right.
$$  
As a matter of convention, we also define
$\fb_{0,P}(m) = C^{*}_r(\Z)$ for $m=0$ and $(0)$ otherwise. We refer to this as the {\it basic ideal function with the data $(q, P)$.} It is easiest to verify this is an abstract ideal function using the criterion given in Note \ref{note:pos_gen}. In fact,  once you verify the case of $q=1$, the general case follows by considering the subalgebra $\cA < \cC^q < \cC$.  If $J \lhd \CZ$ is the closed ideal associated with the closed subset $P \in \Cl(S^1)$, we also denote $\fb_{q, J} = \fb_{q, P}$. 
\end{example}
Directly from the definition of a closed ideal function, it follows that $\fb_{q, J}$ is minimal among all closed ideal functions with $J < \fc(q)$. Alternatively, this can be verified using Note \ref{note:pos_gen}.

In particular for every $\fc \in \fC$ and every $n \in \N$ we have $\fb_{n,\fc(n)} \preceq \fc$. Taking the join over successive values of $n$ yields better and better approximations for $\fc$ from below:
Thus, we define: 
\begin{defn}
For $\fc \in \fC$ we define the {\it{$n^{\mathrm{th}}$ approximation of $\fc$ from below}} by 
$$
\fc^n = \bigvee_{\substack{m \in \Supp{\fc} \\ 0 \le m \le n}} \fb_{m,\fc(m)} = \bigvee_{\substack{m \in \Crit{\fc} \\ 0 \le m \le n}} \fb_{m,\fc(m)},$$
where $\text{Crit}(c)$ is defined below shortly.
When $\fc(m+1) = \fc^{m}(m+1)$, then $\fb_{m,\fc(m)}$ can be omitted from the above join without changing anything. We thus define {\it{the critical points for $\fc$}} by 
$$\Crit({\fc}) = \{m \in \Supp(\fc) \ | m\ge 0,\  \fc^{m}(m+1) \lneqq \fc(m+1)\}.$$
This explains the notation in the rightmost expression of the above definition.
\end{defn}
\begin{theorem} \label{thm:canonical_decomposition}
(The canonical decomposition theorem) \ 
For every $\fc \in \fC$ we have
$$\fc = \bigvee_{n \in \Crit(\fc)} \fb_{n,\fc(n)} = \bigvee_{n \in \Crit(\fc)} \fc^{n}.$$
We refer to this join as {\it{the Canonical generation of $\fc$}}.
\end{theorem}
\begin{proof}
The right-hand side is dominated by the left-hand side. Conversely, for every $n$, we have $\fc(n) \le \fb_{n,c(n)}(n)$, which accounts for the converse inclusion. 
\end{proof}
\begin{example} \label{eg:join_not_closed}
Using basic ideal functions, it is easy to demonstrate how the join of two closed ideal functions need not be closed. Indeed  
$$(\set{bc}_{1,p} \overline{\vee} \set{bc}_{1,\tau p})(n) = \left\{ \begin{array}{ll} \emptyset & \abs{n} \le 1 \\
\{p,\tau^{-1}p, \ldots , \tau^{-n+2}p\} & n \ge 2 \\
\{\tau^2 p, \tau^3 p, \ldots, \tau^{n} p \} & n \le -2 \end{array} \right. $$
which is not a closed ideal function. In fact one easily checks that the condition $\set{\fc}(1)=\emptyset$ immediately forces $\set{\fc}(n)=\emptyset$ for any closed ideal function. 
\end{example}

\begin{example} \label{eg:dynamical_factors}
For a fixed $q \in \N$ note that 
$$
\fb_{q,\CZ}(m) = \left\{ 
   \begin{array}{ll} 
   \CZ & m = nq, \\
   (0) & q \not{\mid} m  
   \end{array} \right.
$$  
is exactly the ideal function coming from the ``dynamical'' intermediate algebra $\cC^q = C(Y) \rtimes_r \CZ$. Namely the one coming from the $q$-fold factor $X \stackrel{\times q}{\longrightarrow} Y$ of the circle. Note that by Suzuki's theorem \cite{Suz18} mentioned at the beginning of the introduction, any larger ideal function is of the form $\fb_{r,\CZ}$ for some $r|q$. 
\end{example}

\subsection{The algebra \texorpdfstring{$\cA_{\fc}$}{}}
We now turn to define a functor in the other direction $\Psi:\fD \rightarrow \fA$.
\begin{defn} \label{def:Phi}
For $\fc \in \fD$ we define 
$$\cA_{\fc}' = \left\langle e^{inx} \eta_n \ | \ n \in \N, \eta_n \in \fc(n) \right \rangle,$$
As the abstract $*$-algebra, generated by these elements, and set $\Psi(\fc) =  \cA_{\fc} = \overline{\cA_{\fc}'}$ to be its closure. 
\end{defn}
For $\fc \in \fC$, the algebra $\cA'_{\fc}$ assumes a much more explicit form. In fact it follows directly from the conditions $\cone, \ctwo$ that 
\begin{equation} \label{eqn:explicit_Ac} 
\cA_{\fc}' = \left\{\sum_{n=-N}^{N} e^{inx} \eta_n \ | \ N \in \N, \eta_n \in \fc(n) \right\}.
\end{equation}
namely, the collection of all such generalized trigonometric polynomials is closed under $*$ and multiplication. For this reason, it is much easier to work with closed ideal functions $\fc \in \fC$. Moreover, limiting our attention only to closed ideal functions does not limit the generality as our main theorem asserts in particular that $\Psi(\fC) = \Psi(\fD)$.
\begin{prop}
For $\fc \in \fC$, $\cA_{\fc}=\Psi(\fc)$ is an intermediate $C^{*}$-subalgebra $\CZ < \cA_{\fc} < \cC$. Moreover, the following conditions are equivalent:
\begin{enumerate}
    \item \label{itm:cq} $\fc(q) = \CZ$ for some $0 \ne q \in \Z$,
    \item \label{itm:standard} $\cA_{\fc} = \cC^q$ comes from a  dynamical factor $(X,T) \stackrel{\times q}{\longrightarrow} (X,T^q)$,
    \item \label{itm:cross} $\cA_{\fc} = \cB \rtimes_r \Z$ for some subalgebra $\cB$ of $C(X)$,
    \item \label{itm:intersection} $\cA_{\fc} \cap C(X) \ne \C$,
\end{enumerate}
    \end{prop}
 \begin{proof}
 $\cA_{\fc} = \Psi(\fc)$ is by definition a closed $*$-algebra and since $\fc(0)=\CZ$, we have $\CZ < \cA_{\fc}$.

The implications
(\ref{itm:standard}) $\Rightarrow$ (\ref{itm:cross}) $\Rightarrow$ (\ref{itm:intersection}) are all obvious. 

To prove (\ref{itm:cq}) $\Rightarrow$ (\ref{itm:standard}) note that if $\fc(q)=\CZ$ for some $q \ne 0$ then $\fb_{q,\CZ} \preceq \fc$. By Suzuki's theorem (see Example \ref{eg:dynamical_factors}), we know that $\fc$ is necessarily an intermediate dynamical algebra of the form $\cC^r$ for some $r|q$. If $q$ were the minimum positive value such that $\fc(q)=\CZ$, then $q=r$. 

To show that (\ref{itm:intersection}) $\Rightarrow$ (\ref{itm:cq}), let $f\in \mathcal{A}_{\fc}\cap C(X)$ be a non-constant function, with a non-trivial Fourier coefficient $\mu(fe^{-iqx})\ne 0$ for some $0 \ne q \in \mathbb{Z}$. Since $f \in \cA_{\fc}$, we can find an element $\sum_{n=-N}^Ne^{inx}\eta_n\in \mathcal{A}_{\fc}'$ such that 
\[\left\|f-\sum_{n=-N}^Ne^{inx}\eta_n\right\|< \epsilon.\]
Where the precise value of $\epsilon$ will be determined later. 
By possibly adding zero coefficients, we may assume that $q\in \{-N, -N+1,\ldots, N-1, N\}$. Multiplying the expression inside the norm, from the left by $e^{-iqx}$ and applying the conditional expectation $\CE$, we obtain:

\[\left\|\mu(fe^{-iqx})- \eta_{q} \right\|=\left\|\mu(fe^{-iqx})-\CE\left(\sum_{n=-N}^Ne^{i(n-q)x}\eta_n\right)\right\| < \epsilon\]
Where we used the fact that 
\(\CE\left(\sum_{n \ne q} e^{i(n-q)x}\eta_n\right)=0.\) 
Dividing out by $\mu(f e^{-iqx})$ we obtain 
\[\left\|1-\frac{\eta_{q}}{\mu(fe^{-iqx})}\right\|< 
\frac{\epsilon}{\abs{\mu(f e^{-iqx})}}\]
 
But if $\epsilon < \abs{\mu(fe^{-iqx})}$ this latter expression guarantees that $\eta_{q}$ is an invertible (and in particular non-zero) element in the algebra. And since by definition $\eta_{q} \in \fc(q)$ we conclude that $\fc(q) = \CZ$ proving (\ref{itm:cq}). 
\end{proof}

\subsection{Basic Subalgebras} \label{sec:basic}
We denote the algebra associated with a basic ideal function $\fb_{q, P}$ by $\cA_{q, P}$. It is easy to find a generator for this subalgebra. 
\begin{prop} \label{prop:prim_alg_gen}
If $\phi \in C(S^1)$ is any function with $\cZ \phi := \{t \in S^1 \ | \ \phi(t)=0\} = P$, namely any generator for the ideal of functions vanishing on $P$, then $\cA_{q,P} = \overline{\langle \CZ,e^{iqx} \hat{\phi} \rangle}$, 
the smallest closed subalgebra of $\cC$ containing the element $a = e^{iqx}\hat{\phi}$ and $\CZ$.
\end{prop}
\begin{proof}
Set $\cB = \overline{\langle \CZ,e^{iqx} \hat{\phi} \rangle}$. By definition $\CZ, e^{iqx} \hat{\phi}$ are both contained in $\cA_{q,P}$ and hence $\cB < \cA_{q,P}$. Conversely for any $m \in \N$ applying successively Equation $(3)$ from Lemma \ref{lem:alg_opp} $m$ times we obtain 
$$
a^m = e^{imqx} \left[(\phi \circ \tau^{(m-1)q}) (\phi \circ \tau^{(m-2)q}) \ldots \phi \right]^{\wedge}.
$$
If we denote the function inside the square brackets by $\phi(m)$ we see immediately that $\cZ \phi(m) = P \cup \tau^{-q}P \cup \ldots \cup \tau^{-q(m-1)}P = {\set{{\fb}}}_{q,P}(m)$. And consequently $e^{imqx} \widehat{\phi(m)} \CZ = e^{imqx} \fb_{q,P}(mq) \subset \cB$. Similarly, using Equation $(2)$ from the same lemma, we obtain a similar inclusion for the negative values of $m$. Hence, $\cA_{q, P}' \subset \cB$, and we get the desired inclusion by passing to the closure. 
\end{proof}

\subsection{Ian Putnam's \texorpdfstring{$Y$}{}-orbit breaking subalgebras}\label{Putnam}

Following \cite{Phill}, we recall the following definition:

\begin{definition}
 Let $X$ be a compact metric space and $T: X \to X$ be a homeomorphism. Consider the transformation group 
 $C^*$-algebra $\mathcal{C} = C(X) \rtimes_r \Z$.  For a closed subset $Y \subset X$, we define the 
 $C^*$-subalgebra $\mathcal{C}_Y$ to be
 the $C^*$-subalgebra of $\mathcal{C}$ generated by $C(X)$ and $\lambda_1 C_0(X \setminus Y)$.
 We say that $\mathcal{C}_Y$ is the {\it $Y$-orbit breaking subalgebra} of $\mathcal{C}$.
 \end{definition}
 
Applying the duality involution interchanging the unitary operators $U= M_{e^{it}}$ with $V = \lambda_1$
and $T= R_\alpha$ with $\tau = R_{-\alpha}$, as in  of Remark
\ref{cor:symmetry} above,
we can interpret Putnam's $Y$-orbit breaking subalgebras as follows:

\begin{definition}
 Let $X$ be a compact metric space and $T: X \to X$ be a homeomorphism. Consider the transformation group 
 $C^*$-algebra $\mathcal{C} = C(X) \rtimes_r \Z$. For a closed subset $P \subset S^1$, we define the 
 $C^*$-subalgebra $\mathcal{C}_P$ to be
 the $C^*$-subalgebra of $\mathcal{C}$ generated by $C(S^1)$ and $e^{ix} C_0(S^1\setminus P)$.
 \end{definition} 
 
 Now comparing this definition with our definition of the basic subalgebra $\mathcal{A}_{1, P}$,
 we conclude that $\mathcal{C}_P = \mathcal{A}_{1,P}$, so that in the particular
 case of the irrational crossed product $C^*$-algebra, the basic algebra $\mathcal{A}_{1,P}$
 and Putnam's $Y$-orbit breaking subalgebra is the same object.

\subsection{Ideals in intermediate subalgebras}
\noindent
Let $\cA_{\fc} = \Psi(\fc) \in \fA$, and $J \lhd \cA_{\fc}$ be a closed two-sided ideal. Just like we did in the case of algebras, we can associate with $J$ an ideal valued function $\Phi_{\fc}(J) = \fj: \Z \rightarrow \cI$ by setting
$$\fj(n) = \{\eta \in \fc(n) \ | \ e^{inx} \eta \in J\}.$$ A calculation analogous to the one carried out in Lemma \ref{lem:alg_opp} and Proposition \ref{cor:O_props} yields the following properties for this function. 
\begin{definition}\label{def:j}
    Given $\fc \in \fC$ a function $\fj:\Z \rightarrow \cI$ will be called a {\it{$\fc$-ideal function}} if $\set{\fj}(n) \supset \set{\fc}(n), \forall n \in \Z$. 
  The {\it support} of $\fj$ is the set 
$$
\Supp(\fj) = \{n \in \mathbb{Z} \ | \ \fj(n) \ne (0)\}
= \{n \in \mathbb{Z} \ |\  \set{\fj}(n) \not = S^1\}.
$$
    A $\fc$-ideal function will be called {\it{closed}} if in addition it satisfies the following two conditions:
    \begin{itemize}
    \item [$\cone(\fc)$] $\set{\fj}(-n) = \tau^{n}\set{\fj}(n), \ \forall n \in \Z$,
    \item [$\ctwoR{\fc}$] $\set{\fj}(m+n) \subset \tau^{-m} \set{\fj}(n) \cup \set{\fc}(m)$ for every $m \in \Supp(\fc), n \in \Supp(\fj)$.
\end{itemize}
\end{definition}

\begin{remark} \label{rem:symmetric_j}
Note that the second condition above is stated in a highly non-symmetric form. A longer but more explicit form, in the case where $m,n \in \Supp(\fj)$ would read: 
\begin{align*}&\set{\fj}(m+n) \\&\subset \left(\tau^{-m} \set{\fj}(n) \cup \set{\fc}(m) \right) \cap  
\left(\tau^{-n} \set{\fj}(m) \cup \set{\fc}(n) \right) \cap  
\left(\tau^{-m} \set{\fc}(n) \cup \set{\fj}(m) \right) \cap  
\left(\tau^{-n} \set{\fc}(m) \cup \set{\fj}(n) \right)
 \end{align*}
\end{remark}
\begin{remark} \label{rem:positive_j}
Similar to the case of ideal functions, $\fc$-ideal functions are completely determined by their values on $\N_0 = \N \cup \{0\}$. For a function $\set{\fj}: \N_0 \rightarrow \Cl(S^1)$ satisfying $\set{\fj}(n) \supset \set{\fc}(n), \forall n \in \N$ it is enough to verify the following conditions for $m,n \in \N_0$ with $m \in \Supp(\fc), n \in \Supp(\fj)$ in order to guarantee that the unique extension to $\Z$ forms a legal closed $\fc$-ideal function:
\begin{eqnarray*}
\set{\fj}(m+n) & \subset & \left[\tau^{-m} \set{\fj}(n) \cup \set{\fc}(m)\right] \cap \left[\tau^{-n} \set{\fc}(m) \cup \set{\fj}(n) \right], \\
\set{\fj}(m-n) & \subset & \left[ \tau^{n-m} \set{\fj}(n) \cup \set{\fc}(m) \right] \cap \left[\tau^n \left(\set{\fc}(m) \cup \set{\fj}(n) \right) \right], \forall 0 \le n \le m, \\
\set{\fj}(n-m) & \subset & \left[ \tau^{m-n} \set{\fc}(m) \cup \set{\fj}(n) \right] \cap \left[\tau^m \left(\set{\fc}(m) \cup \set{\fj}(n) \right) \right], \forall 0 \le m \le n,
\end{eqnarray*}
\end{remark}
\begin{proof}
    We extend the function $\fj$ to $\Z$ using $\cone(\fc)$. We must verify condition $\ctwoR{\fc}$ in the three missing cases. In all cases we assume $m,n \in \N_0$ with $m \in \Supp(\fc), n \in \Supp(\fj)$.

\begin{small}
\begin{align*}
\set{\fj}(-m-n) & =  \tau^{m+n}\set{\fj}(m+n) \subset \tau^{m+n} \left(\tau^{-n} \set{\fc}(m) \cup \set{\fj}(n)  \right) = \tau^{m} \set{\fj}(-n) \cup \set{\fc}(-m) & \\
\fj(-m+n) & = \tau^{m-n}\fj(m-n) \subset \tau^{m-n} \tau^{n} (\set{\fc}(m) \cup \set{\fj}(n) = \tau^{m}(\set{\fc}(m) \cup \set{\fj}(n), & 0 \le n \le m \\
\set{\fj}(m-n) & = \tau^{n-m}\fj(n-m) \subset \tau^{n-m} (\tau^{m-n} \set{\fc}(m) \cup \set{\fj}(n)) = \tau^{-m} \fj(-n) \cup \set{\fc}(m), & 0 \le m \le n \\
\end{align*}
\end{small}
As required. 
\end{proof}
Conversely, with any $\fc$-ideal function, we can associate an ideal. 
\begin{definition}
Given any $\fc$-ideal function $\fj$ as above, we define
$$J_{\fj} = \Psi_{\fc}(\fj) = \overline{\left\langle \left. e^{inx} \eta_n \right| n \in \Z, \ \eta_n \in \fj(n) \right \rangle} \lhd \Psi(\fc).$$ 
\end{definition}
\begin{prop} [Properties of $\fc$-ideal functions] \label{prop:prop}
Let $J = \Psi_{\fc}(\fj)$ for some closed $\fc$-ideal function as above. Then 
\begin{itemize}
\item
$\fj(0) = \mathbb{E}_\mu(J) =J \cap C^*(\Z)$.
\item $J = \cA_{\fc}$ if and only if  $\fj(0) = \CZ$, if and only if $\set{\fj}(0) = \emptyset$,
\item $J = (0)$ if and only if 
  $J \cap C^*(\Z) =(0)$,  if and only if 
$\set{\fj}(0) = S^1$.
\item $\set{\fj}(0) \subset \set{\fj}(n), \forall n \in \Z$.
\end{itemize}
In particular $J$ is nontrivial in the sense that  $(0) \lneqq J \lneqq \cA$, if and only if $\emptyset \subsetneq \set{\fj}(0) \subsetneq S^1$.

\begin{proof}
The first three items are clear; for the last item,
it is enough to consider $n \in \Supp(\fj)$. Using $\ctwoR{\fc}$ we obtain for such $n$
$$
\set{\fj}(0) = \set{\fj}(n+(-n)) \subset \tau^{-n}\tau^{n} \set{\fj}(n) \cup \set{\fc}(n) = \set{\fj}(n) \cup \set{\fc}(n) = \set{\fj}(n).
$$
\end{proof}
\end{prop}

\section{The main theorem}

\subsection{Extraction of generalized Fourier coefficients.}~\par
\vskip1mm
\noindent
Let $V_n = \Span\{e^{imx} \ | \ \abs{m}\le n\} < C(X)$ be the space of trigonometric polynomials of degree at most $n$. With respect to the standard $L^2$-metric, let $Q_n:C(X) \rightarrow V_n$ be the orthogonal projection and $V_n^{\perp} = \ker(Q_n)$. The projection $Q_n$ is explicitly given by the formula 
$Q_nf = \sum_{m=-n}^{n} \overbracket{f}(m) e^{imx}$. It is easy to verify that $Q_n$ is a $T$-equivariant map.

It turns out that $Q_n$ can be extended to the whole $\cC$. By abuse of notation, we denote the extended operator by the same name. 
\begin{prop}
\thlabel{extensionofQ} \label{prop:Qn_bounded}
The map $Q_n$ extends naturally to a bounded linear map, $Q_n: \cC \rightarrow \cC$, that restricts to the identity on $\CZ$.  
\begin{proof}
Clearly $Q_{n} \upharpoonright_{C(X)}:C(X)\to C(X)$ is a bounded linear operator. Using \cite[Proposition~1.10]{pisier2003introduction}, we see that $Q_{n}$ is completely bounded. Now, it follows from \cite[Theorem~3.5]{raeburn1989equivariant} that $Q_{n}$ extends uniquely to a completely bounded map, denoted again by $Q_{n}$ from $\mathcal{C}\to \mathcal{C}$ such that
for every $N$
\[Q_{n}\left(\sum_{i=-N}^N f_i\lambda_i\right)\mapsto \sum_{i=-N}^N Q_{n}(f_i)\lambda_i.\]
\end{proof}
\end{prop}

\begin{defn} \label{def:Q}
For every element $a \in \cC$, let $q(a)$ be the smallest $q \in\mathbb{N}$ such that $Q_{q(a)}(a) \not \in \CZ$, with the convention that $q(a) = \infty$ whenever $a \in \CZ$ and $Q_{\infty} = \id$. 
\end{defn}
Explicitly for non-constant  $f \in C(X)$ we have $f \in V_{q(f)-1}^{\perp}$ and 
\begin{equation} \label{eqn:Qq}
f = Q_{q(f)}(f) + g  = e^{-iq(f)x}\overbracket{f}(-q(f)) +\overbracket{f}(0)
+ e^{iq(f)x}\overbracket{f}(q(f)) + g
\end{equation}
with $g \in V_{q(f)}^{\perp}$. A similar decomposition holds for every $a \in \cC$. Using Proposition \ref{prop:Qn_bounded} above, Equation (\ref{eqn:Qq}) can be applied to each Fourier coefficient $f_i = \mathbb{E}(a\lambda_i^*)$ of $a$ separately and $q(a)$ is just the minimum over all $q(f_i)$.  The following two propositions are crucial to our proof. 
\begin{prop} \thlabel{prop:Q}
Let $\cA \in \fA$. Then for every $a \in \cA \setminus C_r^*(\mathbb{Z})$ with $q = q(a)$, we have:
$$Q_{q(a)}(a) = e^{-iqx}\overbracket{a}(-q) +\overbracket{a}(0)
+ e^{iqx}\overbracket{a}(q)  \in \cA.$$
If in addition $a \in J \lhd \cA$, with $J$ 
a closed two-sided ideal, then $Q_{q(a)}(a) \in J$.
\end{prop}
\begin{lemma}
\thlabel{operatorwithnormone}
Let $N\subset\N$ be a finite set with each $r\in N$ satisfying $r>q$.
Then, for every $\epsilon>0$ there exist $n \in \N$ and numbers $\{k_i \in \Z \ | \ 1 \le i \le n\}$ and $\{\beta_i \in \C \ | \ 1 \le i \le n\}$, with $\sum_{i=1}^{n} \abs{\beta_i} = 1$ such that the operator
     $$L_{N,\epsilon}(a)=\sum_{i=1}^n\beta_i\lambda_{k_i} a\lambda_{k_i}^*,$$
has norm at most one and satisfies 
 $L_{N,\epsilon}(e^{iqx})\approx_{\epsilon} e^{iqx}$ and $L_{N,\epsilon}(e^{irx}) \approx_{\epsilon}0$ for every $r\in N$.
Moreover, the operator $L_{N,\epsilon}$ can also be made to satisfy $L_{N,\epsilon}(e^{-iqx})\approx_{\epsilon} e^{-iqx}$, or $L_{N,\epsilon}(e^{-iqx})\approx_{\epsilon} 0$.
 \end{lemma}
\begin{proof} It is enough to prove the claim for $N =\{r\}$. Indeed, if this were the case, then given any finite subset $N=\{r_1,r_2,\ldots,r_n\}$, a standard induction argument shows that $L_{N,\epsilon}$ can be taken to be $L_{r_n,\frac{\epsilon}{n}}\circ L_{r_{n-1},\frac{\epsilon}{n}}\circ\dots\circ L_{r_1,\frac{\epsilon}{n}}$.
We now prove the claim for $N=\{r\}$. Consider the subgroup $H_1\le S^1$ of $r$-th roots of unity. Under the map $x\mapsto x^q$, the image of $H_1$ forms a smaller non-trivial subgroup $H_2\le S^1$ since $q<r$. Given any root of unity $w_j\in H_2$, we can choose some $k_j$ with the property that $e^{ik_j\alpha}$ is approximately an $r$-th root of unity such that $(e^{ik_j\alpha})^q\approx_{\epsilon} w_j$.\\
\noindent
Note that
\begin{equation*}
\label{eq:conjugationoperation}
    \lambda_ke^{iqx}\lambda_k^* = (e^{ik\alpha})^q e^{iqx}\text{ and }
    \lambda_ke^{irx}\lambda_k^* = (e^{ik\alpha})^r e^{irx}
\end{equation*}
Therefore, we see that
$$\lambda_{k_j}e^{iqx}\lambda_{k_j}^*\approx_{\epsilon} w_je^{iqx}$$
$$\lambda_{k_j}e^{irx}\lambda_{k_j}^* \approx_{\epsilon} e^{irx}$$
Multiplying by the scalar $\overline{w_j}$ on both sides, we see that 
$$\overline{w_j}\lambda_{k_j}e^{iqx}\lambda_{k_j}^*\approx_{\epsilon} e^{iqx}$$
$$\overline{w_j}\lambda_{k_j}e^{irx}\lambda_{k_j}^* \approx_{\epsilon} \overline{w_j}e^{irx}$$
It is now easy to see that $L_{r,\epsilon}$ given by
$$L_{r,\epsilon}(a)=\frac{1}{|H_2|}\sum_{w_j\in H_2}\overline{w_j}\lambda_{k_j} a\lambda_{k_j}^*$$ is the required map. This finishes the proof.    
\end{proof}
Note that $L$ is not a positive map on the crossed product, but it is still a completely bounded map with (CB) norm at most $1$.
\begin{proof}[Proof of \thref{prop:Q}]
Given any $a \in \mathcal{A}$, we may assume without loss of generality that the Fourier coefficient at zero is zero, and choose to write it as
$$a=e^{-iqx}a_{-q} +e^{iqx}a_q +g+h,$$
where $g$ is a finitely supported element with Fourier coefficients only supported on $|r|>q$, and $h$ is some element with $\|h\|<\epsilon$. Then choosing the operator $L$ as in \thref{operatorwithnormone}, we either have
\[L(a) = L(e^{-iqx}a_{-q} + e^{-iqx}a_q) + L(g) + L(h)\approx e^{-iqx}a_{-q} + e^{iqx}a_q + 0 + 0,\]
or, we have \[L(a) = L(e^{-iqx}a_{-q} + e^{-iqx}a_q) + L(g) + L(h)\approx 0 + e^{iqx}a_q + 0 + 0.\] Since $L(a)\in\mathcal{A}$, at least one of $e^{-iqx}a_{-q} + e^{iqx}a_q$ or $e^{iqx}a_q$ lies in $\mathcal{A}$, since at least one of the above statements is true for a sequence of epsilons converging to zero. By performing a similar argument with $a^*$ instead of $a$, we can also conclude that at least one of $e^{-iqx}a_{-q} + e^{iqx}a_q$ or $e^{-iqx}a_{-q}$ lies in $\mathcal{A}$. In any case, regardless of which combination of statements is true, we may one way or another obtain that $e^{-iqx}a_{-q} + e^{iqx}a_q\in\mathcal{A}$.\\ A similar argument works for an ideal $J\triangleleft\mathcal{A}$ since $J$ is invariant under the conjugation action of $\mathbb{Z}$.
\end{proof}

We go one step further and extract individual generalized Fourier coefficients. 
\begin{prop}\label{prop:fourier_coef}
Let $\cA \in \fA$ and $a \in \cA$. Then for every $q \in \Z$ we have $e^{iqx} \overbracket{a}(q) \in \cA$. If in addition $a \in J \lhd \cA$, with $J$ a closed two-sided ideal, then $Q_q(a) \in J$. 
\end{prop}
\begin{proof}
First, we claim that it is enough to prove the proposition for the special case $q= \pm q(a)$. Indeed set $q_0 = q(a)$, Proposition \ref{prop:Q} implies that $a_1 := a - Q_{q_0}(a) \in \cA$. Moreover, for $q_1 = q(a_1) = q_0+1$ we have $\overbracket{a_1}(q_1) = \overbracket{a}(q_1)$. Continuing inductively in this form, we can extract all Fourier coefficients of $a$. Thus, we will assume that $q=q(a)$. Since $\CZ < \cA$ by assumption we may assume that $\overbracket{a}(0)=0$. By Proposition \ref{prop:Q} we can replace $a$ by $Q_q(a)$ and thus assume from now on that 
$$a = e^{iqx} \overbracket{a}(q) + e^{-iqx}\overbracket{a}(-q).$$ 
Let us define the derivative of $a$ to be the element 
$$a' = iqe^{iqx} \overbracket{a}(q) - iqe^{-iqx}\overbracket{a}(-q).$$ 
We claim that $a' \in \cA$. Indeed, since $\alpha$ is assumed to be irrational, we can find a sequence such that $\lim_{j \rightarrow \infty} \overline{n_j \alpha} = 0$ where $\overline{n_j \alpha}$ denotes the representative of $n_j \alpha + \Z$ in the interval $(-1/2,1/2]$. Now 
\begin{align*} 
a' & = \lim_{j \rightarrow \infty} \left(\frac{e^{iq(x+\overline{n_j \alpha})}-e^{iqx}}{\overline{n_j\alpha}} \overbracket{a}(q) + \frac{e^{-iq(x+\overline{n_j \alpha})}-e^{-iqx}}{\overline{n_j\alpha}} \overbracket{a}(-q) \right) \\
& = \lim_{j \rightarrow \infty} \frac{\lambda_{n_j} a \lambda_{-n_j} - a}{\overline{n_j \alpha}} 
\end{align*}
This convergence is uniform in $x$ because the exponential functions are twice continuously differentiable. Thus $a' \in \cA$ as claimed. Finally 
$$e^{iqx}\overbracket{a}(q) = \frac{a}{2}+\frac{a'}{2iq} \in \cA, \qquad \qquad e^{-iqx}\overbracket{a}(-q) = \frac{a}{2}-\frac{a'}{2iq} \in \cA.$$
Which completes the proof of the proposition. Again, the same proof works in the case $a \in J \lhd \cA$.
\end{proof}

\subsection{Proof of Theorems \ref{thm:main} and \ref{thm:main_ideal}}
Assume that $\Phi(\fc) < \cA$ for $\fc \in \fD, \cA \in \fA$. By definition this means that $e^{inx} \phi \in \cA$ for every $n \in \Z, \phi \in \fc(n)$; which in turn shows that 
$\phi \in \Psi(\cA)(n)$. Thus $\fc < \Psi(\cA)$.
Conversely suppose $\fc < \Psi(\cA)$ and assume that $\phi \in \fc(n) < \Psi(\cA)(n)$ for some $n \in \Z$. By definition of $\Psi(\cA)$ this means that $e^{inx} \phi \in \cA$. But by definition $\Phi(\fc) = \overline{ \langle e^{inx} \phi \ | \ \phi \in \fc(n) \rangle}$ and the desired inclusion $\Phi(\fc) < \cA$ follows. This verifies (\ref{itm:GC}) from the main theorem, namely that $\Phi$ and $\Psi$ form a Galois connection. Unlike the one in Galois theory, this is a monotone (i.e., order-preserving) Galois connection. However, one naturally obtains an order reversing Galois connection upon composing with the natural identification $\fD \cong \set{\fD}$.

The images $\Psi(\fA) \subset \fD$ and $\Phi(\fD) \subset \fA$ are now called closed ideal functions and closed intermediate algebras, respectively. Statement (\ref{itm:isom}) of the main theorem now follows from standard results about Galois connections, which we recall here: 
\begin{prop} \label{prop:galois_connection}
For a monotone Galois connection as above, the following hold: 
\begin{enumerate}[(i)]
\item \label{itm:gc1} $\fc \preceq \Psi \Phi \fc, \ \forall \fc \in \fD$ and $\cA \ge \Phi \Psi \cA, \forall \cA \in \fA$. 
\item \label{itm:gc2}  $\Phi(\fc_1)  \le \Phi(\fc_2)$ whenever $\fc_1 \preceq \fc_2$ and similarly $\Psi(\cA_1) \preceq \Psi(\cA_2)$ whenever $\cA_1 \le \cA_2$. 
\item \label{itm:gc3} $\Phi \Psi \Phi \fc = \Phi \fc, \ \forall \fc \in \fD$ and $\Psi \Phi \Psi \cA = \Psi \cA, \ \forall \cA \in \fA$. 
\item \label{itm:gc4} $\Phi\upharpoonright_{\Psi \fA}: \Psi \fA \rightarrow \Phi \fD$ is an isomorphism of lattices between the closed elements on both sides. And $\Psi$ is its inverse. 
\end{enumerate}
\end{prop}
\begin{proof}
Applying the Galois connection property to $\Phi(\fc) \le \Phi(\fc)$ and to $\Psi \cA \preceq \Psi \cA$ yields \ref{itm:gc1}. Assume that $\fc_1 \preceq \fc_2$. By \ref{itm:gc1} we have $\fc_1 \preceq \fc_2 \preceq \Psi\Phi \fc_2$, and the defining property of the Galois connection yields $\Phi(\fc_1) \leq \Phi(\fc_2)$. The dual assertion in \ref{itm:gc2} is proved similarly. Now applying \ref{itm:gc1} to $\Phi(\fc) \in \fA$, immediately implies that $\Phi \fc \ge \Phi \Psi \Phi \fc, \ \forall \fc \in \fD$. Similarly \ref{itm:gc1} also implies that $\fc \preceq \Psi \Phi \fc$, and by \ref{itm:gc2} $\Phi \fc \le \Phi \Psi \Phi \fc$. Equality $\Phi \fc = \Phi \Psi \Phi \fc, \ \forall \fc \in \fD$ follows which is one clause in \ref{itm:gc3}. The dual clause is proved identically. The last property \ref{itm:gc4} is a direct consequence of \ref{itm:gc3}. 
\end{proof}
The property \ref{itm:gc2} above is referred to as the monotonicity property for obvious reasons. In our case, it is evident from the definitions of $\Phi,\Psi$, but as we saw above, it also follows formally from the defining property of a monotone Galois connection. We now turn to the less formal statements of the main theorem.

Let $\cA \in \fA$ be an intermediate algebra and denote $\fc = \Psi \cA$. By the above Proposition, $\cA \ge \Phi \fc$. We must establish the opposite inclusion to prove (\ref{itm:perfect}). Thus let $a \in \cA$, by Proposition \ref{prop:fourier_coef} we know that $e^{inx} \overbracket{a}(n) \in \cA, \ \forall n \in \Z$, which by definition means that $\overbracket{a}(n) \in \fc(n), \ \forall n \in \Z$. By definition of $\Phi$ this means that $e^{inx} \overbracket{a}(n) \in \Phi(\fc), \ \forall n \in \Z$. Thus $Q_n(a) = \sum_{k=-n}^{n} e^{ikx} \overbracket{a}(k) \in \Phi(\fc)$ and by Fej\'{e}r's Theorem \ref{thm:Fejer} we deduce that 
$$a = \lim_{N \rightarrow \infty} \frac{1}{N} \sum_{n = 0}^{N-1} Q_n (a) \in \Phi(\fc),$$ showing that the Galois connection is perfect on the algebra side, and finishing the proof of (\ref{itm:perfect}).  

On the other side of the Galois connection, Proposition \ref{prop:galois_connection} implies $\fc \le \Psi \Phi \fc$ for every $\fc \in \fD$. To prove (\ref{itm:closed_fc}), we must show that the ideal functions satisfying the opposite inclusion are exactly these in $\fC$. From Proposition \ref{cor:O_props}, we already know that the conditions $\cone,\ctwo$, defining the class $\fC$, are necessary for an ideal function to be closed. It remains to show that these conditions are also sufficient. 

Let $\fc \in \fC$, set $\cA = \Phi \fc$ and assume, towards a contradiction, that $\eta \in \Psi \cA (n) \setminus \fc(n)$ for some $n$. We may assume that $n \in \N$ is the minimal such index. By definition $\eta \in \Psi(\cA)(n)$ means that $e^{inx} \eta \in \cA$. 

Since $\fc \in \fC$ 
$$\cA_{\fc}' = \left\{ \sum_{k = -K}^{K} e^{ikx} \eta_k \ | \ K \in \N, \eta_k \in \fc(k)\right\}.$$
is an algebra, closed under the $*$-operation and by definition $\cA = \Phi(\fc)$ is its closure in $\cC$. Thus for every $\epsilon > 0$ we can find an approximation 
$$
    \norm {e^{inx} \eta - \sum_{k = -K}^{K} e^{ikx} \eta_{k}} = \norm {(\eta - \eta_k) -\sum_{\substack{k = -K\\ k \ne n}}^{K} e^{i(k-n)x} \eta_{k}}<\epsilon
$$
Applying the 
conditional expectation $\mathbb{E}_{\mu}$ on both sides we obtain that
\[\norm {\mathbb{E}_{\mu}(\eta - \eta_k) -\mathbb{E}_{\mu}\left(\sum_{\substack{k = -K\\ k \ne n}}^{K} e^{i(k-n)x} \eta_{k}\right)}<\epsilon\]
Since $\eta_k$ falls in the multiplicative domain of $\mathbb{E}_{\mu}$ and $\mathbb{E}_{\mu}(e^{i(k-n)x})=0$ for $k\ne n$, we obtain that
$\norm{\eta - \eta_k} < \epsilon$. Since $\epsilon>0$ was arbitrary and the ideal $\fc(n) \lhd \CZ$ is closed, we obtain the desired contradiction. This completes the proof of Theorem \ref{thm:main}. \qed

We now turn to the proof of Theorem \ref{thm:main_ideal}. It is quite similar to the proof of Theorem~\ref{thm:main}, but we include all the details. Fix $\fc \in \fC$ and let $\cA = \Phi(\fc)$. That $\Phi_{\fc}: \fD_{\fc} \rightarrow \fJ_{\cA}$ and $\Psi_{\cA}:\fJ_{\cA} \rightarrow \fD_{\fc}$ form a monotone Galois connection follows as before. This proves (\ref{itm:GCj}), as well as (\ref{itm:isomj}), using Proposition \ref{prop:galois_connection}. 

To prove (\ref{itm:perfectj}), namely that the connection is perfect on the ideal side, we have to show that $J \leq \Phi_{\fc}(\fj)$ for every $J \in \fJ_{\cA}$ and $\fj=\Psi_{\cA}(J)$. Indeed, by Proposition \ref{prop:fourier_coef}, together with every $a \in J$ we have $e^{inx} \overbracket{a}(n) \in J$ for every $n \in \Z$. Hence by the definition of $\fj = \Psi_{\cA}(J)$ we know that   $\overbracket{a}(n) \in \fj(n), \ \forall n \in \Z$. We conclude using Fej\'{e}r approximations: 
$$a = \lim_{N \rightarrow \infty} \frac{1}{N} \sum_{n = 0}^{N-1} Q_n (a) \in \Phi_{\fa}(\fj).$$  

To prove (\ref{itm:closed_fc}) we must show that the $\fc$-ideal functions satisfying $\Psi_{\cA} \Phi_{\fc} \fj = \fj$ are exactly those satisfying conditions $\cone(\fc), \ctwoR{\fc}$ of Definition \ref{def:j}. Set $J = \Phi_{\fc}(\fj)$. In view of Lemma \ref{lem:alg_opp}, the conditions $\cone(\fc)$ and $\ctwoR{\fc}$ follow directly from the respective requirements that $(e^{inx}\hat{\phi})^* \in J$ and $e^{inx}\hat{\phi} e^{imx} \hat{\psi} \in J$, whenever $e^{inx} \hat{\phi} \in J, e^{imx}\hat{\psi} \in \cA$. Namely, they follow from $J$ being $*$-closed right ideal. It remains to show that if $\fj$ satisfies the conditions $\cone(\fc), \ctwoR{\fc}$ then $\Psi_{\cA} \Phi_{\fc}(\fj) \preceq \fj$, where by property \ref{itm:gc1} of Galois connections the other inclusion is automatic for every $\fj$. So let us assume by way of contradiction that $\eta \in \Psi_{\cA} J (n) \setminus \fj(n)$ for some $n \in \N$. 

By conditions $\cone(\fc), \ctwoR{\fc}$ 
$$J' = \left\{ \sum_{k = -K}^{K} e^{ikx} \eta_k \ | \ K \in \N, \eta_k \in \fj(k)\right\}.$$
is a $*$-closed right ideal of $\cA$. Hence, it is also automatically a left ideal. $J = \Phi_{\fc}(\fj)=\overline{J'}$ is by definition the closure. Thus for every $\epsilon > 0$ we can find an approximation 
$$
    \norm {e^{inx} \eta - \sum_{k = -K}^{K} e^{ikx} \eta_{k}} = \norm {(\eta - \eta_k) -\sum_{\substack{k = -K\\ k \ne n}}^{K} e^{i(k-n)x} \eta_{k}}<\epsilon
$$
Applying the 
conditional expectation $\mathbb{E}_{\mu}$ on both sides we obtain that
\[\norm {\mathbb{E}_{\mu}(\eta - \eta_k) -\mathbb{E}_{\mu}\left(\sum_{\substack{k = -K\\ k \ne n}}^{K} e^{i(k-n)x} \eta_{k}\right)}<\epsilon\]
Since $\eta_k$ falls in the multiplicative domain of $\mathbb{E}_{\mu}$ and $\mathbb{E}_{\mu}(e^{i(k-n)x})=0$ for $k\ne n$, we obtain that
$\norm{\eta - \eta_k} < \epsilon$. Since $\epsilon>0$ was arbitrary and the ideal $\fj(n)$ is closed, we obtain the desired contradiction and conclude the proof of Theorem \ref{thm:main_ideal}. \qed

\section{Structural results about subalgebras} \label{sec:structure}
\subsection{The closed join formula}\label{sub:closed_join}
In this section, we finally establish a closed formula for the closed join
of two algebras $\cA_1, \cA_2 \in \mathcal{A}$:

\begin{prop} \label{prop:closed_join}
Let $\mathcal{A}_1$ and $\mathcal{A}_2$ be two intermediate algebras with associated 
ideal functions $\fc_1,\fc_2 \in \fC$,
so that $\cA_i = \Phi(\fc_i)$. 
Then 
$$
\Psi \circ \Phi (\fc_1 \overline{\vee} \fc_2) = \Psi(\cA_1 \vee \cA_2).
$$
We call this closed ideal function {\it{the closed join}} of $\fc_1$ and $\fc_2$. We denote it by $\fc_1 \vee \fc_2$. A more explicit expression for the closed join $\fd =\fc_1 \vee \fc_2$
is given by
\begin{small}
\begin{gather*}
\set{\fd}(n) =\\
 \bigcap \tau^{-\sum_{k=2}^{s} j_k} \set{\fc}_{\epsilon}(j_1)  \cup \left(\tau^{-\sum_{k=3}^{s} j_k} \set{\fc}_{\epsilon+1}(j_2)\right) \cup \ldots \cup \left(\tau^{-j_s} \set{\fc}_{\epsilon+s-2}(j_{s-1})\right) \cup \set{\fc}_{\epsilon+s-1}(j_s),
\end{gather*}
\end{small}
where it is understood that $\set{\fc}_{h}$ is either  $\set{\fc}_1$ or $\set{\fc}_{2}$, depending on the parity of $h$. The intersection is taken over all possible values $s \ge 2$, $\epsilon \in \{1,2\}$ and $\{j_1,j_2,\ldots, j_s \} \subset \Z$ such that $\sum_{k=1}^s j_k = n$.
\end{prop}
\begin{proof}
For this let $\fc_1,\fc_2 \in \fC$ be two closed ideal functions with corresponding intermediate algebras $\cA_i = \Phi(\fc_i)$. Since $\fc_1 \overline{\vee} \fc_2 > \fc_i, \forall i \in \{1,2\}$ we have 
$\Phi(\fc_1 \overline{\vee} \fc_2) > \cA_i, \forall i \in \{1,2\}$. Hence $\Phi(\fc_1 \overline{\vee} \fc_2) > \cA_1 \vee \cA_2$ so that $\Psi \circ \Phi(\fc_1 \overline{\vee} \fc_2) > \Psi(\cA_1 \vee \cA_2)$. Conversely for any $n \in \Z$ let $\phi_i \in C(S^1)$ be a non-negative function with $\cZ \phi_i = \{t \in S^1 \ | \ \phi_i(t) = 0\} =  \set{\fc}_i(n)$. 
By definition of the algebra $\cA_i = \Phi(\fc_i)$ we have $e^{inx} \hat{\phi}_i \in \cA_i$. Hence $e^{inx} (\hat{\phi}_1 + \hat{\phi}_2) \in \cA_1 \vee \cA_2$. But $\cZ (\phi_1 + \phi_2) = \set{\fc}_1(n) \cap \set{\fc}_2(n) = (\set{\fc}_1 \overline{\vee} \set{\fc}_2)(n)$. 
Thus, $\Phi(\fc_1 \overline{\vee} \fc_2) < \cA_1 \vee \cA_2$ and applying $\Psi$ we obtain $\Psi \circ \Phi(\fc_1 \overline{\vee} \fc_2) < \Psi(\cA_1 \vee \cA_2)$. These two estimates yield the first equality of our proposition. 

Now let $\set{\fd}(n)$ be the ideal function defined by the right-hand side formula claimed in Proposition \ref{prop:closed_join}. That is
\begin{small}
\begin{align} \label{eqn:def_d}
& \set{\fd}(n) = \bigcap Q^{(n,s)}_{\epsilon}(j_1,j_2, \ldots, j_s), \qquad {\text{where}} \\
\nonumber & Q^{(n,s)}_{\epsilon}(j_1,j_2, \ldots, j_s):= \\
\nonumber & \qquad  \tau^{-\sum_{k=2}^{s} j_k} \set{\fc}_{\epsilon}(j_1)  \cup \left(\tau^{-\sum_{k=3}^{s} j_k} \set{\fc}_{\epsilon+1}(j_2)\right) \cup \ldots \cup \left(\tau^{-j_s} \set{\fc}_{\epsilon+s-2}(j_{s-1})\right) \cup \set{\fc}_{\epsilon+s-1}(j_s)
\end{align}
\end{small}
We claim that 
$$\set{\fc}_1 \overline{\vee} \set{\fc}_2 \supset \set{\fd} \supset \set{\fc}_1 \vee \set{\fc}_2.$$ 
Indeed, clearly $Q^{(n,s)}_{\epsilon}(n,0,\ldots,0) = \set{\fc}_{\epsilon}(n)$ for every $n \in \Z$ and $\epsilon \in \{1,2\}$. Since by definition $\left(\set{\fc}_1 \overline{\vee} \set{\fc}_2\right)(n) = \set{\fc}_1(n) \cap \set{\fc}_2(n)$ the first inclusion follows.

To establish the second inclusion let us fix $s,\epsilon$ and numbers $j_k$ such that $\sum_{k=1}^{s} j_k = n$. Consider the product 
\begin{align*}
& (e^{i j_1 x} \hat{\phi}_1) \cdot (e^{i j_2 x} \hat{\phi}_2) \ldots (e^{i j_s x} \hat{\phi}_s) = \\
& e^{inx} \left[(\phi_1 \circ \tau^{j_2+j_3+\ldots+j_s}) (\phi_2 \circ \tau^{j_3+\ldots + j_s}) \cdot \ldots \cdot (\phi_{s-1} \circ \tau^{j_s}) \phi_s \right]^{\wedge} =: e^{inx} \hat{\phi} \in \cA_1 \vee \cA_2.
\end{align*}
Here, the $\wedge$ sign indicates that we have taken the Fourier transform on the function inside the square brackets. We name this function $\phi$. The equality is obtained by an iterated application of the multiplication formula in Lemma \ref{lem:alg_opp}(\ref{itm:c3}). 

If the functions $\phi_k(t) \in C(S^1)$ are chosen to be non-negative with $\cZ(\phi_k) = \set{\fc}_{\epsilon + k -1}(j_k)$, 
this element is inside $\cA_1 \vee \cA_2$. 
Since
$\cZ(\phi) = Q^{s}_{\epsilon}(j_1,j_2,\ldots, j_s)$, we obtain $Q^{s}_{\epsilon}(j_1,\ldots,j_s) \supset (\set{\fc}_1 \vee \set{\fc}_2)(n)$. Taking finite sums over various possible choices of $s,\epsilon,_1,\ldots, j_s$ and then passing to the limit yields the second desired containment. 
It follows from this that 
$$\Phi (\set{\fc}_1 \overline{\vee} \set{\fc}_2) = \Phi(\set{\fd}) = \Phi(\set{\fc}_1 \vee \set{\fc}_2).$$ Thus it would be enough to show that $\fd$ is closed in order to show the desired equality $\set{\fd} = \set{\fc}_1 \vee \set{\fc}_2$.

Consider first $\cone$.  
We see that
\begin{align*}
&\tau^n Q^{(n,s)}_{\epsilon}(j_1,j_2, \ldots, j_s) =  \tau^{j_1} \set{\fc}_{\epsilon}(j_1) \cup \ldots \cup \left(\tau^{\sum_{k=1}^{s-1} j_k} \set{\fc}_{\epsilon+s-1}(j_{s-1})\right) \cup \tau^{n} \set{\fc}_{\epsilon+s}(j_s)\\
& =  \set{\fc}_{\epsilon}(-j_1) \cup \ldots \cup \left(\tau^{\sum_{k=1}^{s-2} j_k} \set{\fc}_{\epsilon+s-1}(-j_{s-1})\right) \cup \left(\tau^{\sum_{k=1}^{s-1} j_k} \set{\fc}_{\epsilon+s}(-j_s) \right) \\
& = Q^{(-n,s)}_{\epsilon+s}(-j_s,-j_{s-1},\ldots,-j_1)
\end{align*}
Plugging this into the definition of $\set{\fd}$ in Equation (\ref{eqn:def_d}) we conclude that $\set{\fd}(-n) = \tau^{n} \set{\fd}(n)$ which is exactly $\cone$. 

Turning to $\ctwo$, consider  $Q^{(m+n,s)}_{\epsilon}(j_1, \ldots, j_r,j_{r+1},\ldots, j_s)$ where $\sum_{k=1}^r j_k = n$, $\sum_{k=r+1}^s j_k = m$. Then we have 
\begin{align*}
&\set{\fd}(m+n)\\&\subset Q^{(m+n,s)}_{\epsilon}(j_1,j_2, \ldots, j_s)  \\
& =\tau^{-m} \left[\tau^{-\sum_{k=2}^{r} j_k} \set{\fc}_{\epsilon}(j_1)  \cup \ldots \cup  \set{\fc}_{\epsilon+r}(j_r)\right] \cup \left[ \tau^{-\sum_{k=r+2}^{s} j_k} \set{\fc}_{\epsilon+r+1}(j_{r+1}) \cup \ldots \cup \set{\fc}_{\epsilon+s}(j_s) \right]  \\
& = \tau^{-m} Q^{(n,r)}_{\epsilon}(j_1,\ldots,j_r) \cup Q^{(m,s-r)}_{\epsilon+r+1}(j_{r+1},\ldots, j_s)
\end{align*}
Applying the same calculation to $Q^{(m+n,s+1)}_{\epsilon}(j_1, \ldots, j_r,0,j_{r+1},\ldots, j_s)$ we obtain an inclusion 
$\set{\fd}(m+n) \subset \tau^{-m} Q^{(n,r)}_{\epsilon}(j_1,\ldots,j_r) \cup Q^{(m,s-r)}_{\epsilon+r+1}(j_{r+1},\ldots, j_s)$ and taking the intersection over these two we have now
$$\set{\fd}(m+n) \subset \tau^{-m} Q^{(n,r)}_{\epsilon}(j_1,\ldots,j_r) \cup Q^{(m,s-r)}_{\epsilon'}(j_{r+1},\ldots, j_s).$$
Now the parameters $\epsilon, r, j_1,\ldots, j_r$ appearing in the first part of this expression and $\epsilon', s-r,j_{r+1},\ldots, j_s$ appearing in the second part are completely independent of each other. Taking the intersection over all of these yields the desired result:
$$\set{\fd}(m+n) \subset \tau^{-m} \set{\fd}(n) \cup \set{\fd}(m)$$
which is exactly $\ctwo$. This completes our proof. 
\end{proof}

\subsection{Residual algebras}\label{sub:residual}

The special class of residual intermediate subalgebras was singled out in the introduction as algebras admitting especially nice properties. Recall that an ideal function $\fc \in \fC$ (resp. a $\fc$-ideal function $\fj$) are called {\it{residual}} if $\set{\fc}(n)$ (resp. $\set{\fj}(n)$) has an empty interior for every $n$ in the support. We also refer to the corresponding intermediate algebra (resp ideal) as residual. We denote by $\fC^{r}$ (resp. $\fI_{\fc}^{r}$) the sub-lattices of residual objects.

\begin{note}
The notation could be confusing. Residual algebras are large, but we are talking about small sets at the level of closed subsets of the circle. In fact, $\fc$ is residual exactly when the complement $S^1 \setminus \set{\fc}(n)$ is residual (in the standard Baire category sense of the word), for every $n \in \Supp(\fc)$.
\end{note}

\begin{proof}[Proof of \thref{cor:residual}]
It is clear that the class of residual algebras $\fC^{r}$ is closed under arbitrary joins. The fact that it is closed under finite intersections follows from the (finite) Baire category theorem. This establishes Properties (\ref{itm:vee}) and (\ref{itm:wedge}). 

Any ideal function's support $\Supp(\fc)$ is symmetric by property $\cone$. In the residual case, it is closed under addition by $\ctwo$ because the union of two closed nowhere dense sets is still nowhere dense. This establishes Property (\ref{itm:supp_gp}). 

Now let $J \lhd \cA \in \fA$ be a nontrivial closed two-sided ideal with associated $\fc$-ideal function $\fj$ and set $M := \set{\fj}(0) \subset S^1$. When we say that $J$ is nontrivial, we mean that $(0) \lneqq J \lneqq \cA$, which means that $\emptyset \subsetneq M \subsetneq S^1$, as in Proposition (\ref{prop:prop}).

As $\set{\fc}(n) \subset \set{\fj}(n)$ for every $n \in \mathbb{Z}$, it is clear that $\Supp(\fj) \subset \Supp(\fc)$. The reverse inclusion also holds. Indeed by $\ctwoR{\fc}$ we have, for $k \in \Supp(\fc)$, that  $\set{\fj}(k) = \set{\fj}(k+0) \subset  \set\fc(k) \cup \tau^{-k}\set\fj(0) \subsetneq S^1$, where the proper containment 
follows since  $\set{\fj}(0)$ is a proper closed subset of $S^1$ and $\set\fc(k)$ is closed with an empty interior. Thus $\Supp(\fc) = \Supp(\fj)$, establishing Property (\ref{itm:supp_ideal}).  
 
To prove Property (\ref{itm:residual_ideal}) we first argue that $M = \set{\fj}(0)$ is nowhere dense. By (\ref{itm:supp_gp}) and (\ref{itm:supp_ideal}) $\Supp(\fc) = \Supp(\fj) = r \Z$ for some $r \geq 1$ is a subgroup of $\mathbb{Z}$.  Recall that by Proposition (\ref{prop:prop}) $M \subset \set{\fj}(k), \forall k \in \Z$. Now applying $\ctwoR{\fc}$ we have for every $k,n \in \Supp(\fc)$
 \begin{eqnarray} \label{eqn:M}
\nonumber     M & \subset & \set{\fj}(k) = \set{\fj}(n + k - n) \subset \tau^{-n-k} \set{\fj}(-n) \cup \set{\fc}(n+k) = \tau^{-k} \set{\fj}(n) \cup \set{\fc}(n+k) \\ 
     & \subset & \tau^{-k} \set{\fj}(n) \cup \tau^{-k} \set{\fc}(n) \cup \set{\fc}(k) = \tau^{-k} \set{\fj}(n) \cup \set{\fc}(k),
 \end{eqnarray}
 where the last equality uses, again, the fact that $\set{\fj}(n) \supset \set{\fc}(n)$. 

Taking intersection over all $n \in \Supp(\fc)$ and using the fact that $\bigcap_{n \in \Z} \set{\fj}(n) = M$, by Proposition \ref{prop:prop}, we obtain 
    \begin{equation} \label{eqn:M2}
    M \subset \set{\fj}(k) \subset \tau^{-k} \left( \bigcap_{n \in \Z} \set{\fj}(n) \right) \cup \set{\fc}(k) = \tau^{-k} M \cup \set{\fc}(k)
    \end{equation}

Let us denote by 
$\Omega = \bigcup_{m,n \in \Supp(\fc)} \tau^{m} \set{\fc}(n)$. 
This is a meager $\tau^r$-invariant subset of $S^1$. Now, using the last step, we have 
$$M \setminus \Omega \subset \tau^{-r}\left( M \setminus \Omega\right).$$ 
Namely $M \setminus \Omega$ is $\tau^r$-invariant. It cannot be dense because $M$ is a proper closed subset of $S^1$. Hence, it has to be empty by the minimality of $\tau^r$. Thus $M \subset \Omega$, particularly $M$, is meager. 

Now Equation (\ref{eqn:M2}) implies that $\set{\fj}(k) \subset \tau^{-k}M \cup \set{\fc}(k) \subset  \Omega$, for every $k \in \Supp(\fj) = \Supp(\fc)$. In particular, $\set{\fj}(k)$ is closed and nowhere dense for every $k \in \Supp(\fj)$. This completes the proof of Property (\ref{itm:residual_ideal}). 

Finally, towards a proof of Property (\ref{itm:center_free}), suppose that $a$ is in the center of $\mathcal{A}$. We first show that $a \in \CZ$.
For every $n$ we have $a \lambda_n= \lambda_1(a\lambda_n)\lambda_1^*$, hence
$$
\mathbb{E}(a\lambda_n) = \mathbb{E}(\lambda_1(a\lambda_n)\lambda_1^*)
= \lambda_1\mathbb{E}(a\lambda_n) \lambda_1^* = 
\mathbb{E}(a\lambda_n) \circ T.
$$
By ergodicity, each Fourier coefficient $\mathbb{E}(a\lambda_n)$ is a constant, and we conclude that $a \in \CZ$.
Thus $a = \hat{\xi}$ for some $\xi \in C(S^1)$.

Let $n \in \Supp(\fc)$ be such that $\set{\fc}(n)$ is nowhere dense and let $\phi \in C(S^1)$ be such that $\hat{\phi}$ is a generator of the ideal $\fc(n)$. Or in other words $\mathcal{Z}\phi = \{t \in S^1 \ | \ \phi(t) = 0 \} = \set{\fc}(n)$. 
We then have, using Lemma \ref{lem:alg_opp}(\ref{itm:c1})
$$
e^{inx}\widehat{\phi \xi}= e^{inx}\hat{\phi} \hat{\xi} = \hat{\xi} e^{inx}\hat{\phi}
= e^{inx} \widehat{\left[(\xi \circ \tau^n) \phi \right]},
$$
so on the open dense set $S^1 \setminus \set{\fc}(n)$
we get $\xi \circ \tau^n = \xi$. Using the ergodicity 
of $\tau^n$, we conclude that $\xi$ is a constant. This establishes Property (\ref{itm:center_free}) and completes the proof of Corollary (\thref{cor:residual}). 
\end{proof}
\begin{note}
According to Property (\ref{itm:wedge}) in the above proof, the intersection of any finite collection of residual subalgebras is still residual. This does not generalize to countable intersections, though. Consider, for example, the basic algebras 
$\mathcal{A}_i:=\cA_{1,\{p_i\}}$ with $\{p_i \ | \ i \in \N\} \subset S^1$ a dense countable set. Finite intersections of these algebras are residual, and in fact, they are still basic:
$$\Psi \left(\bigcap_{i=1}^n \mathcal{A}_i\right) = \fb_{1,\{p_1,p_2,\ldots,p_n\}}.$$
But clearly 
$$\bigcap_{i \in \N} \mathcal{A}_i = C(X).$$
\end{note}

\subsection{Small algebras}\label{sub:small}
This section is dedicated to studying small subalgebras, whose basic properties are summarized in \thref{cor:small}. 

\begin{proof}[Proof of \thref{cor:small}]
Let $\cA \in \cA$ and $\fc = \Psi(\cA)$. By our main structure theorem, we know that, as a $\CZ$ module 
$$\cA = \overline{\oplus_{n \in \Supp(\fc)} \fc(n)}.$$
Thus, it is clear that $\Supp(\fc)$ is finite if and only if the algebra is finite-dimensional as a module over $\CZ$. These are the small subalgebras and (\ref{itm:small_fd}) directly follows. 

Since for residual subalgebras $\Supp(\fc)$ is a (nontrivial) subgroup of $\Z$, residual subalgebras can never be small. Now let $\fb_{q, P}$ be a basic ideal function and $\cA_{q, P}$ the associated algebra (see Example \ref{eg:prim} and Subsection \ref{sec:basic}). If $P$ has an empty interior, then $\cA_{q, P}$ is residual and hence not small. Conversely, if $P$ has a nonempty interior, by the minimality of $\tau: S^1 \rightarrow S^1$, there is some $N$ such that $\set{\fb}_{q,P}(nq) = P \cup \tau^{-q}P \cup \ldots \cup \tau^{q(1-n)}P = S^1$ for every $n \ge N$. By the explicit formula for the basic ideal function in Example \ref{eg:prim}, this immediately implies that $\Supp(\fb_{q, P})$ is finite and hence $\cA_{q, P}$ is small. This proves (\ref{itm:basic_dichotomy})

If $P=\set{\fc}(q)$ has an empty interior for some $q \ne 0$, then $\cA$ contains the basic residual algebra $\cA_{q, P}$ as a subalgebra. Hence, $\cA$ itself must be residual, contrary to the assumption that it is small. This proves (\ref{itm:nonempty_int}). 
The following examples will demonstrate parts (\ref{itm:small_center}) and (\ref{itm:join_not_small}).
\end{proof}

\begin{example}\label{exa:center}
Let $J = S^1 \setminus P = (-\pi/10, \pi/10)$ and suppose that
$ \pi/10 < \alpha < 2\pi/10$. Then any function $\phi \in C(S^1)$
such that $\phi(t) = \phi(t+ \alpha)$ for $t \in J$
and $\phi(t) = \phi(t - \alpha)$ for $t \in \tau J$,
is a central element of the algebra $\mathcal{A}_{1,P}$.
\end{example}

\begin{proof}
Clearly $P \cup \tau P = S^1$ so that every element $a$ of 
$\mathcal{A}_{1,P}$ has the form
$$
a = e^{-ix}\hat\psi_{-1} + \hat\psi_0 + e^{ix}\hat\psi_1,
$$
with $\psi_{-1}\upharpoonright_{\tau P} = 0 =
\psi_1\upharpoonright_{P}$.
In order for an element $\hat\phi \in C^*(\mathbb{Z})$ to be central in 
$\mathcal{A}_{1,P}$ it should satisfy the equation
$\hat\phi a = a\hat\phi$ for every $a \in \mathcal{A}_{1,P}$.
Writing this equation more explicitly, we get
$$
e^{-ix}(\hat\phi \circ \tau^{-1})\hat\psi_{-1} + \hat\phi \hat\psi_0 + e^{ix}
(\hat\phi \circ \tau)\hat\psi_1  =
e^{-ix}\hat\psi_{-1}\hat\phi + \hat\psi_0\hat\phi + e^{ix}\hat\psi_1 \hat\phi.
$$
From this we deduce that on $J$ we have $\phi = \phi \circ \tau^{1}$
and on $\tau J$ we have $\phi = \phi \circ \tau^{-1}$.
By our assumption the sets $J$,  $\tau^{-1}J$, 
$\tau J$ and $\tau^2J$ are pairwise disjoint, and it follows that any $\phi$ which satisfies these conditions 
will be a central element of $\mathcal{A}_{1, P}$.
\end{proof}

The intersection of two residual algebras is residual. In contrast, the following example shows that the join of two small algebras need not be small. 
\begin{example}
Consider two subsets $P$,$Q\subset  S^1$ with non-empty interior such that $P\cap Q=\emptyset$. Let  $\fa = \fb_{1,P}, \fb=\fb_{1,Q}$ be the corresponding ideal functions. It is easy to see that $(\set{\fa}\vee\set{\fb})(1)=\emptyset$ as it is contained in both $\set{\fa}(1)=P$ and $\set{\fb}(1)=Q$. In other words, the join of the corresponding algebra contains $e^{ix}$ and $C_r^*(\mathbb{Z})$, and is thus the entire crossed product $C(S^1)\rtimes_r\mathbb{Z}$.
\end{example}
\subsection{About the simplicity of intermediate algebras. }\label{sub:simple}

It is well known that the algebras 
$\mathcal{C}^q =C(X_q) \rtimes_r \mathbb{Z}$ are simple (each $(X,R_{q\alpha})$ being minimal).
We would like to determine which of the intermediate algebras are simple. 
To approach this question, consider a general intermediate algebra $\mathcal{A}_{\fc}$.
The collection of closed sets $\{\set{\fc}(n) : |n| \geq 1\}$
either has the finite intersection property with 
$$
Q = Q(\fc) := \bigcap_{|n| \geq 1} \set{\fc}(n)  \not = \emptyset,
$$
or it does not. In the latter case
there is an $N \geq 1$ such that $\bigcap_{1 \leq |n| \leq N} \set{\fc}(n) = \emptyset$.

\begin{prop}\label{prop:not-simple}
Let $\fc$ be an ideal function such that
$Q = \bigcap_{|n| \geq 1} \set{\fc}(n)  \not = \emptyset$,
then the algebra $\mathcal{A}_{\fc}$ is not simple.
In particular, for the algebra $\mathcal{A}= \mathcal{A}_{\fb_{q, P}}$ corresponding to the basic ideal functions $\fb_{q, P}$,
we have $Q = P \cap \tau^q P$ so that $\mathcal{A}$ is not simple whenever this set is not empty.
\end{prop}

\begin{proof}

{\bf {$Q \not = \emptyset \Rightarrow \mathcal{A} = \mathcal{A}_{\fc}$ is not simple}}.

Let 
\begin{equation}\label{I}
I = \{a \in \mathcal{A} : \overbracket{a}(0) \upharpoonright_{Q} =0\}.
\end{equation}
It is  easy to check that this is a proper ideal:
 Given $d \in I$ and $a \in  \mathcal{A}$
we need to show that both $\overbracket{ad}(0)= \mathbb{E}_{\mu}(ad)$ and
$\overbracket{da}(0)=\mathbb{E}_\mu(da)$ 
correspond to functions that vanish on $Q$.
Let $a' = \sum_{|n| \leq N} e^{inx}\hat{\phi}_n$ 
and $d' = \sum_{|k|\leq N} e^{ikx}\hat{\psi}_k $
be good Fej\'{e}r approximations to 
$a$ and $d$ respectively.
Applying the conditional expectation, we get the following:
\begin{align} \label{is_ideal}
\nonumber \mathbb{E}_\mu(ad) &  \approx \
\mathbb{E}_\mu\left( (\sum_{|n|\leq N}e^{inx}\hat{\phi}_n)(\sum_{|k|\leq N} e^{ikx}\hat{\psi}_k)\right)\\
& = \sum_{0 <|n|\leq N}(\widehat{\phi_n \circ \tau^{-n}})\hat{\psi}_{-n}
 + \widehat{\phi_0 \cdot \psi_0}.
\end{align}
Now the first summoned corresponds to a function on $C(S^1)$ which vanishes on $Q$,
as,  for $n  \not= 0$, we have that $\psi_{-n}$ vanishes on $\set{\fc}(-n)$, 
$\phi_n \circ \tau^{-n}$ also vanishes on $\tau^{n} \set{\fc}(n) = 
\set{\fc}(-n)$, and $Q \subset \set{\fc}(-n)$. 
The second summoned corresponds to the function $\phi_0\psi_0$
which belongs to $I$.
Since $I$ is closed, we conclude that indeed $ad \in I$.
A similar computation shows that also $da \in I$.

For the last assertion, we note that for
$\fb_{q,P}$ we have $Q = P \cap \tau^q P$.
\end{proof}

Is the condition $Q(\fc) \not = \emptyset$ also a necessary
condition for $\mathcal{A}_{\fc}$ to be simple ?
We only have a partial answer to this question in Proposition \ref{prop:necessary} below.

\begin{lemma}
If $I < \mathcal{A}_{\fc}$ is a proper ideal, then $\mathcal{B} = \overline{C^*(\mathbb{Z}) + I}$
is a $C^*$-subalgebra: $C^*(\mathbb{Z}) < \mathcal{B} < \mathcal{A}_{\fc}$,
and $I $ is a proper ideal in $\mathcal{B}$.
\end{lemma}

\begin{proof}
Easy to check.
\end{proof}

\begin{definition}
\thlabel{genrideal}
We say that an ideal $I < \mathcal{A}_{\fc}$ is {\it generating} when $\mathcal{A}_{\fc}
= \overline{C^*(\mathbb{Z}) + I}$.
\end{definition}

See also \cite[Corollary 1.5.8]{Pedersen} for the following lemma.

\begin{lemma}\label{lem:noclosure}
If $I$ is a generating ideal for $\mathcal{A}_{\fc}$ then we actually have:
$$
\mathcal{A}_{\fc}
=C^*(\mathbb{Z}) + I.
$$
\end{lemma}

\begin{proof}
By Theorem \ref{thm:main_ideal} we have $I = I_{\fj}$
for some ideal function $\fj$.
Let $I_0 = I \cap C^*(\Z)$.
It then follows that $\mathbb{E}_\mu( I) \subset I_0$.
Let $c  \in \mathcal{A}_{\fc}$, then we have 
$$
c = \lim (a_n + b_n),
$$
with $a_n \in C^*(\mathbb{Z})$ and $b_n \in I$.
Now
$$
 \mathbb{E}_\mu(c) = \lim \ \mathbb{E}_\mu(a_n + b_n)
= \lim \ (a_n + \mathbb{E}_\mu(b_n)).
$$
It follows that $c -  \mathbb{E}_\mu(c) = \lim(b_n - \mathbb{E}_\mu(b_n))$.
As, for each $n$, both $b_n$ and $\mathbb{E}_\mu(b_n)$ are in $I$ and $I$ is closed,
we conclude that $c -  \mathbb{E}_\mu(c)  \in I$, whence $c \in C^*(\Z) + I$.
\end{proof}
\begin{cor}
Given any ideal $I\triangleleft\mathcal{A}_{\fc}$, $C_r^*(\mathbb{Z})+I=\overline{C_r^*(\mathbb{Z})+I}$.
\begin{proof} It follows from \thref{genrideal} that $I$ is a generating ideal for $\overline{C_r^*(\mathbb{Z})+I}$. Since $\overline{C_r^*(\mathbb{Z})+I}$ is an intermediate algebra of the form $\mathcal{A}_{\fd}$ for some $\fd\le \fc$, the claim follows from Lemma~\ref{lem:noclosure}.
\end{proof}
\end{cor}

\begin{remark}\label{rem:specialI}
We note that the ideal $I < \mathcal{A}_{\fc}$ (defined in Equation~(\ref{I}) of Proposition \ref{prop:not-simple})  is 
 of the form $I = I_{\fj}$ with $\set{\fj}(0)=Q$ and $\set{\fj}(n) = \set{\fc}(n)$ for $n \ne0$, hence it is generating. Let us denote such an ideal by $I^{\cA}_{Q}$.
If $Q_1 \subset Q$ is a smaller closed subset, then Formula~(\ref{is_ideal}) shows that $I^{\cA}_{Q_1} \lhd \cA$ as well. It is easy to check that if $Q_1$ is taken to be a point, we obtain a maximal ideal (of codimension $1$) inside $\cA$.
\end{remark}

\vspace{.3cm}

\begin{prop}
Let $\mathcal{A}_{\fc}$ be an intermediate algebra with a generating proper ideal $I < \mathcal{A}_{\fc}$;
then  $Q(\fc) = \bigcap_{|n| \geq 1}\set{\fc}(n) \not=\emptyset$.
Conversely, if $Q(\fc)  \not=\emptyset$ then 
$I = \{a \in \mathcal{A} : \overbracket{a}(0) \upharpoonright_{Q(\fc)} =0\}$
is a proper generating ideal of $\mathcal{A}_{\fc}$.
\end{prop}

\begin{proof}
Let $I_0 = I \cap C^*(\Z)$ and let $I = I_{\fj}$ as in Theorem
\ref{thm:main_ideal}.
It then follows that $\mathbb{E}_\mu( I) \subset I_0$. 

Now suppose to the contrary that $Q = \bigcap_{|n| \geq 1} \set{\fc}(n)= \emptyset$.
As we observed above there is an $N \geq 1$ such that $\bigcap_{1 \leq |n| \leq N} \set{\fc}(n) = \emptyset$.
Denoting $U_n = S^1 \setminus \set{\fc}(n)$, we have $S^1 = \bigcup_{1 \leq |n| \leq N} U_n$.
Let $\{\xi_n\}_{1 \leq |n| \leq N}$ be a partition of unity subordinate to the open cover
$\{U_n\}_{1 \leq |n| \leq N}$; i.e. $0 \leq \xi_n \leq 1$, $\sum_{1 \leq |n| \leq N} \xi^2_n = 1$,
and $\supp \xi_n \subset U_n$.
Let $c = \sum_{1 \leq |n| \leq N} e^{inx}\hat{\xi}_n$.

By Lemma \ref{lem:noclosure} we have $c = \hat{\phi}_0 + d$, with
$\hat{\phi}_0 \in C^*(\Z)$ and $d \in I$.
Now $0 = \mathbb{E}_\mu(c)  = \mathbb{E}_\mu(\hat{\phi}_0 + d) =
\hat{\phi}_0 + \mathbb{E}_\mu(d)$, hence $\hat{\phi}_0 = - \mathbb{E}_\mu(d)   \in I_0 \subset I$
and therefore also $c = \hat{\phi}_0 + d \in   I$.
We now conclude that  
$\mathbb{E}_\mu(c^*c) = \sum_{1  \leq |n| \leq N} \hat{\xi}_n^2 \in I$.
But $\sum_{1  \leq |n| \leq N} \xi_n^2 =1_{S^1}$,
 contradicting our assumption that $I$ is a proper ideal.

\vspace{.3cm}

The other direction follows from Proposition \ref{prop:not-simple} and Remark \ref{rem:specialI}. 
 
\end{proof}

\begin{prop}\label{prop:necessary}
Suppose  $\mathcal{A}_{\fc}$ is not simple.
Let $I < \mathcal{A}_{\fc}$ be a proper ideal
(which we can assume to be maximal) and
let  $\mathcal{B} = C^*(\mathbb{Z}) + I \subset \mathcal{A}_{\fc}$.
Then $\mathcal{B}$ is a sub-$C^*$-algebra, and  $I$ is a generating ideal for $\mathcal{B}$.
Moreover, $\mathcal{B} = \mathcal{A}_{\fd}$ for an ideal function $\fd$ with
$\set{\fd}(n) \supset \set{\fc}(n)$ for every $n \in \mathbb{Z}$,
and such that $Q(\fd) = \bigcap_{1 \leq |n|} \set{\fd}(n)\not = \emptyset$.
\end{prop} 

As was mentioned above, the algebras 
$\mathcal{C}^q =C(X_q) \rtimes_r \mathbb{Z}$ are simple.
We next describe additional examples of simple intermediate subalgebras. These include the basic subalgebras 
corresponding to the ideal functions $\fb_{q,p} = \fb_{q,\{p\}}$, with $0 \ne q \in \mathbb{N}, p \in S^1$,
thus generalizing (in the context of $\mathcal{C}$) 
\cite[Proposition 11.3.21]{Phill} (in view of Subsection 
\ref{Putnam}).

\begin{prop}
\thlabel{residualsimplicity}
Assume that $\cA \in \fA^r$ is a residual intermediate algebra with $\fc = \Psi(\cA)$ such that $\fc(q)=\{p\}$ is a singleton for some $0 \ne q \in \Z$. Then, $\cA$ is simple.
\end{prop}
\begin{proof}
  Assume towards a contradiction that $J \lhd \cA$ is a non-trivial ideal and let $\fj$ be the corresponding $\fc$-ideal function.  Since $\cA$ is residual, everything from the proof of \thref{cor:residual} in sub-Section \ref{sub:residual} applies. We thus adopt the notation from that proof and, in particular, set $M = \set\fj(0)$. As in that proof, the non-triviality of the ideal implies $M \ne \emptyset$ and $M \ne S^1$. 
    
    Using Equation (\ref{eqn:M2}) from the proof of \thref{cor:residual}, we conclude that
    \begin{equation} \label{eqn:M3}
        \emptyset \subsetneq M \cap \set\fc(k) \subsetneq 
        \set\fc(k), \qquad \forall k \ne 0.
    \end{equation}
    Indeed if $M \cap \set\fc(k) = \emptyset$ then Equation 
    (\ref{eqn:M2}) would read $M \subset \tau^{-k}M$. By minimality of $\tau$, the closed subset $M \subset S^1$ is either empty or everything, which is not the case since $J$ is a nontrivial ideal. Similarly if $\set\fc(k) \subset M$ we can apply  Equation (\ref{eqn:M2}) with $-k$ instead of $k$, to obtain 
    $$M \subset \tau^k M \cup \set{\fc}(-k) = \tau^k \left(M \cup \set{\fc}(k)\right) = \tau^k M,$$ 
    leading to the same contradiction. Taking $k = q$, we obtain a contradiction to the assumption that 
    $\set\fc(q) = \{p\}$ is a singleton.    
\end{proof}

\begin{remark}
By \cite[Theorem 12.2.5]{Phill},
when a closed subset $P \subset S^1$
is such that $\tau^n P \cap P =\emptyset$
for all $n \not = 0$, the algebra
$\mathcal{A}_{1,P}$ is {\it centrally large} (hence large). It is simple and infinite dimensional by 
\cite[Theorem 12.2.6]{Phill}. Presently, we don't know how to prove this
result using our methods.
\end{remark}

\section{Crossed products of non \texorpdfstring{$C^*$}{}-simple groups}\label{sec:not-simple}
Many of the methods we used to obtain a complete classification of intermediate algebras in Theorem \ref{thm:main} are pretty specific to the system we considered there. 
This section is dedicated to the proof of Theorem (\ref{notcrossedprod}), where we show that the mere existence of intermediate algebras that do not come from dynamical factors holds for a much more general 
class of crossed products of the form $\cA \rtimes_r\Gamma$. For example,
this is the case whenever $\cA$ admits a faithful $\Gamma$-invariant state, $\Gamma$ is not $C^*$-simple and does not admit a normal subgroup isomorphic to
$\Z/2\Z$.

\subsection{Crossed products}
We briefly describe the construction of the crossed product for unital $C^*$-algebras in general and refer the reader to \cite{BO:book} for more details.~\par
\vskip1mm
\noindent
Let $\Gamma$ be a discrete group and $\mathcal{A}$ be an unital $C^*$-algebra. An action of $\Gamma$ on $\mathcal{A}$ is a group homomorphism $\alpha$ from $\Gamma$ to the group of $*$-automorphisms on $\mathcal{A}$. A $C^*$-algebra equipped with a $\Gamma$-action is called a $\Gamma$-$C^*$-algebra. Suppose that $\mathcal{A}$ is a unital $\Gamma$-$C^*$-algebra. Let $\pi:\mathcal{A} \to \mathbb{B}(\mathcal{H})$ be a faithful $*$-representation. Let $\ell^2(\Gamma,\mathcal{H})$ be the space of square summable $\mathcal{H}$-valued functions on $\Gamma$, i.e.,
\[\ell^2(\Gamma,\mathcal{H})=\left\{\xi:\Gamma\to \mathcal{H}\text{ such that }\sum_{t\in\Gamma}\|\xi(t)\|_{\mathcal{H}}^2<\infty.\right\}\]
There is an action $\Gamma\curvearrowright \ell^2(\Gamma,H)$ by left translation:
\[\lambda_s\xi(t):=\xi(s^{-1}t), \xi \in \ell^2(\Gamma,\mathcal{H}), s,t \in \Gamma\]
Let $\sigma$ be the $*$-representation \[\sigma:\mathcal{A} \to \mathbb{B}(\ell^2(\Gamma,\mathcal{H}))\] defined by \[\sigma(a)(\xi)(t):=\pi(t^{-1}a)\xi(t), a \in \mathcal{A}\]
where $\xi \in \ell^2(\Gamma,\mathcal{H})$, $t \in \Gamma$.
The reduced crossed product $C^*$-algebra $\mathcal{A}\rtimes_{r}\Gamma$ is the closure in $\mathbb{B}(\ell^2(\Gamma,\mathcal{H}))$ of the subalgebra generated by the operators $\sigma(a)$ and $\lambda_s$. Note that $\lambda_s\sigma(a)\lambda_{s^{-1}}=\sigma(s.a)$ for all $s \in \Gamma$ and $a \in \mathcal{A}$.

It follows from the construction that $\mathcal{A}\rtimes_r\Gamma$ contains $C_{\lambda}^*(\Gamma)$ as a $C^*$-sub-algebra. The reduced crossed product $\mathcal{A}\rtimes_r\Gamma$ comes equipped with a $\Gamma$-equivariant canonical conditional expectation $\mathbb{E}:\mathcal{A}\rtimes_r\Gamma\to\mathcal{A}$ defined by 
\[\mathbb{E}\left(\sigma(a_s)\lambda_s\right)=\left\{ \begin{array}{ll}
0 & \mbox{if $s\ne e$}\\
\sigma(a_e) & \mbox{otherwise}\end{array}\right\}\]
It follows from \cite[Proposition~4.1.9]{BO:book} that $\mathbb{E}$ extends to a faithful conditional expectation from $\mathcal{A}\rtimes_r\Gamma$ onto $\mathcal{A}$. Moreover, for a subgroup $H\le\Gamma$, there is a faithful conditional expectation $\mathbb{E}_H:\mathcal{A}\rtimes_r\Gamma\to\mathcal{A}\rtimes_rH$ (see~\cite[Remark~3.2]{khoshkam1996hilbert}) defined by 
\[\mathbb{E}_H\left(\sigma(a_s)\lambda_s\right)=\left\{ \begin{array}{ll}
\sigma(a_s)\lambda_s & \mbox{if $s\in H$}\\
0 & \mbox{otherwise}\end{array}\right\}\]

Explicit examples of intermediate algebras not coming from any factor $Z$ of the form $Y\xrightarrow[]{}Z\xrightarrow[]{}X$, associated with an inclusion $C(X)\subset C(Y)$ were presented in \cite[Proposition~2.6]{Suz18}. 
Ideals have also been used to create intermediate algebras for tensor product inclusion which do not split (see, e.g.,~\cite[Corollary~3.4]{zacharias2001splitting}).\\
This section demonstrates that the ideals obstruct the intermediate algebras being crossed products. Given a unital $\Gamma$-$C^*$-algebra, we show that we can 
create an algebra that is not a crossed product in a canonical way, as long as the reduced $C^*$-algebra is not simple, under some assumptions on the unital $C^*$-algebra $\mathcal{A}$.  

\subsection{Intermediate \texorpdfstring{$C^*$}{}-algebra associated to an ideal}
Let $\Gamma$ be a discrete group acting continuously on an unital $C^*$-algebra $\mathcal{A}$ by $*$-automorphisms. Assume that there exists a normal subgroup $N\triangleleft\Gamma$ such that $C_r^*(N)$ is not $\Gamma$-simple, i.e., $C_r^*(N)$ has a $\Gamma$-invariant closed two sided ideal. Let $J$ be a two-sided, non-trivial, closed $\Gamma$-invariant ideal in $C_r^*(N)$. We observe that if $N\triangleleft\Gamma$ is such that $C_r^*(N)$ is not $\Gamma$-simple, then $\Gamma$ is a not a $C^*$-simple group (see~\cite[Theorem~1.1]{AK}).
Using $J$, we can build a non-trivial closed two-sided ideal $I$ inside $C_r^*(\Gamma)$.
\begin{lemma}
\thlabel{biggeridealassociatedtoasmallerone}
Let $N\triangleleft\Gamma$ be such that $C_r^*(N)$ is not $\Gamma$-simple. Let $J$ be a two-sided, non-trivial, closed $\Gamma$-invariant ideal in $C_r^*(N)$. Then,
\[I=I_J=\overline{\text{Span}\left\{\eta a: \eta\in J, a \in C_r^*(\Gamma)\right\}}\] is a non-trivial closed two sided ideal of $C_r^*(\Gamma)$ which contains $J$.
\begin{proof}
Clearly, $I$ is closed in $\|.\|$ and also, under addition. 
By construction, $I$ is closed under right multiplication. To show that $I$ is closed under the left multiplication, it is enough to show that $\lambda(s)\eta a\in I$ for any $\eta\in J$ and $a\in C_r^*(\Gamma)$. Again writing $\lambda(s)\eta=\eta_s\lambda(s)$ and using the $\Gamma$-invariance of $J$, 
\[\lambda(s)\eta a=\left(\lambda(s)\eta\lambda(s)^*\right)\lambda(s)a=\eta_s\lambda(s)a\in I.\]
This automatically implies that $I$ is a two-sided ideal of $C_r^*(\Gamma)$.
All that needs to be established now is that $I$ is non-trivial. Since $J\ne 0$, $I\ne 0$. Let us show that $I\ne C_r^*(\Gamma)$. Towards a contradiction, let us assume otherwise. Then, given $0 < \epsilon < 1$, we can find $\eta_1,\eta_2,\ldots,\eta_k\in I$ and $a_1,a_2,\ldots,a_k\in C_r^*(\Gamma)$ such that
\[\left\|\lambda(e)-\sum_{i=1}^k\eta_ia_i\right\|<\epsilon.\]
Let $\mathbb{E}_N: C_r^*(\Gamma)\to C_r^*(N)$ denote the canonical conditional expectation from $C_r^*(\Gamma)$ onto $C_r^*(N)$. Note that $\mathbb{E}_N$ sends every element in $N$ to itself and sends each element of $\Gamma \setminus N$ to zero. Applying $\mathbb{E}_N$ to the above inequality and using the fact that $J$ falls in the multiplicative domain of $\mathbb{E}_N$, we see that
\[\left\|\lambda(e)-\sum_{i=1}^k\eta_i\mathbb{E}_N(a_i)\right\|<\epsilon<1.\]
This forces $\sum_{i=1}^k\eta_i\mathbb{E}_N(a_i)\in J$ to be an invertible operator. However, $J\triangleleft C_r^*(N)$ is a non-trivial ideal. This contradicts our assumption that $I=C_r^*(\Gamma)$. Therefore, we are done since $I\ne C_r^*(\Gamma)$.
\end{proof}
\end{lemma}
Given an ideal $I\triangleleft C_r^*(\Gamma)$, let us associate the following operator space $\cA_I$ to $I$ defined by
\[\mathcal{A}_I=\text{Span}\left\{\eta \tilde{a}\eta': \eta,\eta'\in I, ~\tilde{a}\in\cA\rtimes_r\Gamma \right\}\]
Since $I$ is a closed two-sided ideal, $\cA_I$ is closed under the $*$-operation.
We claim that $\overline{\cA_{I}}$ is a $C^*$-algebra. It is enough to show that $\overline{\cA_I}$ is closed under multiplication. Let us denote by $\cA_{\rtimes_{r,\text{alg}}}\Gamma$ the collection of finite sums of the form $\sum_{i=1}^n a_{s_i}\lambda(s_i)$, or more formally
\[\cA\rtimes_{r,\text{alg}}\Gamma=\left\{\sum_{s\in F}a_s\lambda(s)\ | \ F\subset\Gamma, |F|<\infty, ~a_s\in \mathcal{A}\right\}\]
Also, recall that $\cA\rtimes_{r,\text{alg}}\Gamma$ is norm dense inside $\cA\rtimes_{r}\Gamma$. If there exists a $\Gamma$-invariant state $\varphi$ on $\mathcal{A}$, then the map $\mathbb{E}_{\varphi}:\cA\rtimes_{r,\text{alg}}\Gamma\to C_r^*(\Gamma), a\lambda(s)\mapsto\varphi(a)\lambda(s)$ extends to a well-defined map at the level of $\cA\rtimes_{r}\Gamma$ (see e.g., \cite[Exercise~4.1.4]{BO:book}). 
\begin{prop}
\thlabel{algebra}
$\overline{\cA_I}$ is a $C^*$-algebra. Moreover,
$\mathcal{B}=C_r^*(\Gamma)+\overline{\mathcal{A}_{I}}$ is an intermediate $C^*$-algebra.
\end{prop}

\begin{proof}
To show that $\overline{\cA_I}$ is a $C^*$-algebra, it is enough to show that $\overline{\cA_I}$ is closed under multiplication. Towards that end, let us choose elements of the form $\eta_1 \tilde{a_1}\eta'_1,\eta_2 \tilde{a_2}\eta'_2\in \mathcal{A}_I$ with $\eta_1,\eta'_1, \eta_2, \eta'_2\in I$ and $\tilde{a_1},\tilde{a_2}\in\mathcal{A}\rtimes_r\Gamma$. Clearly, \[\eta_1\tilde{a_1}\eta_1'\eta_2\tilde{a_2}\eta_2'=\eta_1\left(\tilde{a_1}\eta_1'\eta_2\tilde{a_2}\right)\eta_2'\in\mathcal{A}_I.\]
It now follows by a standard limiting argument that $\overline{\mathcal{A}_I}$ is closed under multiplication. 
Let us first check that $\mathcal{B}$ is norm closed. If $\{a_n+b_n\}_{n\in \mathbb{N}}\in \mathcal{B}$ is such that $a_n\in C_r^*(\Gamma)$, $b_n\in \overline{\mathcal{A}_I}$ and $a_n+b_n\to c$. then $\E_{\varphi}(a_n+b_n)\to\E_{\varphi}(c)$. Therefore, $a_n+\E_{\varphi}(b_n)\to\E_{\varphi}(c)$. As a result, we see that $b_n-\E_{\varphi}(b_n)\to c-\E_{\varphi}(c)$. Since $b_n\in \overline{\mathcal{A}_{I}}$ and $\E_{\varphi}\left(\overline{\mathcal{A}_{I}}\right)\subset \overline{\mathcal{A}_{I}}$, we see that $b_n-\E_{\varphi}(b_n)\in \overline{\mathcal{A}_{I}}$ and therefore, $c-\E_{\varphi}(c)\in\overline{\mathcal{A}_{I}}$. Hence, $c=\E_{\varphi}(c)+\left(c-\E_{\varphi}(c)\right)\in \mathcal{B}$.
Since $I$ is an ideal of $C_r^*(\Gamma)$, it follows that 
$\lambda(s)\mathcal{A}_I\subset \mathcal{A}_I$ and $\mathcal{A}_I\lambda(s)\subset \mathcal{A}_I$ for all $s \in \Gamma$. Hence, $\mathcal{B}$ is closed under multiplication.
\end{proof}
\begin{remark}
\label{rem:equalspanclosure}
It is also true that
\[\overline{\mathcal{A}_I}=\overline{\text{Span}\left\{\eta a \eta': \eta,\eta'\in I,~a\in\cA\right\}}\]
It is clear that
\[\overline{\text{Span}\left\{\eta a \eta': \eta,\eta'\in I,~a\in\cA\right\}}\subseteq\overline{\text{Span}\left\{\eta \tilde{a} \eta': \eta,\eta'\in I,~\tilde{a}\in\cA\rtimes_r\Gamma\right\}}\]
For the reverse inclusion, we observe that for $\eta,\eta'\in I$, $\eta(a\lambda(s))\eta'=\eta a(\lambda(s)\eta')$ and hence, the claim follows. \qed
\end{remark}
\begin{remark} We can also show that $\cB$ is closed in $\|.\|$ without taking the help of $\mathbb{E}_{\varphi}$.
Let us denote $\mathcal{D}=\overline{\mathcal{B}}$. Since $\mathcal{B}$ is closed under multiplication, so is $\mathcal{D}$. Consequently, $\mathcal{D}$ is a $C^*$-algebra. Moreover, $\overline{\mathcal{A}_I}$ is a closed two-sided ideal of $\mathcal{D}$. Using \cite[Theorem~3.1.7]{murphy1990}, we see that $\overline{\mathcal{A}_I}+C_r^*(\Gamma)$ is a $C^*$-subalgebra.
\end{remark}
\begin{prop}\thlabel{interm}
\thlabel{notcrossedprod-Ga}
Let $\Gamma$ be a discrete
group and $\mathcal{A}$, a $\Gamma$-$C^*$-algebra. Let $I$ be a non-trivial closed $\Gamma$-invariant two-sided ideal in $C_r^*(\Gamma)$.
Let $\varphi$ be a faithful $\Gamma$-invariant state on $\mathcal{A}$.
Let $\mathcal{B}=C_r^*(\Gamma)+\overline{\mathcal{A}_I}$.
Then, $\mathcal{B}\cap\mathcal{A}
=\mathbb{C}$. 
 \end{prop}
\begin{proof}
It follows from \thref{algebra} that $\mathcal{B}$ is an intermediate $C^*$-algebra. Let us now suppose that $\varphi$ is a $\Gamma$-invariant state on $\mathcal{A}$ and $\mathbb{E}_{\varphi}$, the associated conditional expectation onto $C_r^*(\Gamma)$. We first claim that $\mathbb{E}_{\varphi}\left(\overline{\mathcal{A}_I}\right)\subset \overline{\mathcal{A}_I}$. For this to hold, it is enough to show that $\E_{\varphi}\left(\mathcal{A}_{I}\right)\subset\mathcal{A}_{I}$ after which the claim would follow by the density of $\mathcal{A}_{I}$ inside $
\overline{\mathcal{A}_{I}}$ and the continuity of $\E_{\varphi}$.
Towards that end, for an element of the form $\sum_{i=1}^n\eta_i\tilde{a_i}\eta'_i\in \mathcal{A}_{I}$, we see that
\[\E_{\varphi}\left(\sum_{i=1}^n\eta_i\tilde{a_i}\eta'_i\right)=\sum_{i=1}^n\eta_i\E_{\varphi}(\tilde{a_i})\eta'_i\]
Since $\eta'_i\in I$ and $\E_{\varphi}(\tilde{a_i})\in C_r^*(\Gamma)$, we see that $\E_{\varphi}(\tilde{a_i})\eta'_i\in I$ and hence, 
\[\sum_{i=1}^n\eta_i\E_{\varphi}(\tilde{a_i})\eta'_i=\sum_{i=1}^n\eta_i\mathbf{1}_{\mathcal{A}}\E_{\varphi}(\tilde{a_i})\eta'_i\in \mathcal{A}_I.\]
We now show that $\overline{\mathcal{A}_I}\cap \mathcal{A}={0}$. Assume that $0\ne a\in \overline{\mathcal{A}_I}\cap\mathcal{A}$. By replacing $a$ with $a^*a$, we can assume that $a\ge 0$. In particular, $\varphi(a)>0$. 
Let $0<\epsilon<1$ be given. We can find an element of the form $\sum_{i=1}^n\eta_i\tilde{a_i}\eta'_i\in \mathcal{A}_I$ such that 
\[\left\|a-\sum_{i=1}^n\eta_i\tilde{a_i}\eta'_i\right\|<\varphi(a)\epsilon\]
Applying the conditional expectation $\E_{\varphi}$, we see that 
\[\left\|\varphi(a)-\sum_{i=1}^n\eta_i\E_{\varphi}(\tilde{a_i})\eta'_i\right\|<\varphi(a)\epsilon\]
Let us observe that \[\sum_{i=1}^n\eta_i\E_{\varphi}(\tilde{a_i})\eta'_i\in I,\] where $I$ is a non-trivial ideal in $C_r^*(\Gamma)$. Hence, 
\[\left\|1-\frac{\sum_{i=1}^n\eta_i\E_{\varphi}(\tilde{a_i})\eta'_i}{\varphi(a)}\right\|<\epsilon<1\]
Since $\sum_{i=1}^n\eta_i\E_{\varphi}(\tilde{a_i})\eta'_i\in I$, it follows that $\frac{\sum_{i=1}^n\eta_i\E_{\varphi}(\tilde{a_i})\eta'_i}{\varphi(a)}\in I$. Now, $\frac{\sum_{i=1}^n\eta_i\E_{\varphi}(\tilde{a_i})\eta'_i}{\varphi(a)}$ being close to $1$ is an invertible operator in $I$. Hence, $I=C_r^*(\Gamma)$. This is a contradiction. Therefore, $a=0$. 

It is clear that $C_r^*(\Gamma)\cap \cA=\mathbb{C}$. Let us now show that $\cA\cap \mathcal{B}=\mathbb
{C}$. Let $\tilde{a}\in \mathcal{B}\cap\cA$. Then, we can find $a \in C_r^*(\Gamma)$ and $b\in \overline{\cA_{I}}$ such that $\tilde{a}=a+b$. Hence, $b=\tilde{a}-a$. Applying the canonical conditional expectation $\E$, we get that $\E(b)=\tilde{a}-\tau_0(a)$. In particular, $\Tilde{a}=\E(b)+\tau_0(a)$. Therefore, $b=\E(b)+\tau_0(a)-a$.
Since $b\in \overline{\mathcal{A}_{I}}$ and $\E_{\varphi}(\overline{\mathcal{A}_{I}})\subset\overline{\mathcal{A}_{I}}$, we see that $b-\E_{\varphi}(b)\in \overline{\mathcal{A}_{I}}$. Now, \begin{align*}b-\E_{\varphi}(b)&=\E(b)+\tau_0(a)-a-\E_{\varphi}\left(\E(b)+\tau_0(a)-a\right)\\&= \E(b)-\varphi\left(\E(b)\right)\in \cA\end{align*}
Since $\overline{\mathcal{A}_I}\cap \cA=\{0\}$ and $(b-\E_{\varphi}(b))\in \overline{\mathcal{A}_I}$, it follows that $b=\E_{\varphi}(b)$. As a consequence, we see that $\Tilde{a}=a+b=a+\E_{\varphi}(b)$. Applying the canonical conditional expectation on both sides, we see that $$\Tilde{a}=\E(\Tilde{a})=\E(a+\E_{\varphi}(b))=\tau_0\left(a+\E_{\varphi}(b)\right)\in \mathbb{C}.$$
\end{proof}
For the group $\mathbb{Z}/2\mathbb{Z}$ with two elements, it is possible to construct an intermediate algebra that is not a crossed product in a canonical way without using the ideals from $C_r^*(\mathbb{Z}/2\mathbb{Z})$. 
\begin{example}[Intermediate algebra for $\mathbb{Z}/2\mathbb{Z}$]
\label{ex:toycase}
Let $\Gamma=\mathbb{Z}/2\mathbb{Z}
= \{e,s\}$ and $X$, a compact Hausdorff $\Gamma$-space which admits a $\Gamma$-invariant measure with full support. Assume that $X$ has more than two points. Let $\mathcal{A}=C(X)$. Let $$J=\left\{a_e+a_s\lambda(s): a_e,a_s\in\mathbb{C},~a_e+a_s=0\right\}$$ be the non-trivial augmentation ideal.  Let $\mathcal{B}= C_r^*(\Gamma)+\overline{\mathcal{A}_J}$. In the light of \thref{algebra} and \thref{interm}, it is enough to show that $\mathbb{E}(\mathcal{B})\ne\mathbb{C}$. Since $|X|>2$, there exist $x_1$ and $x_2\in X$ such that $x_1\not\in \Gamma x_2$. 
Let $U$ and $V$ be two open neighborhoods containing $x_1$ and $\Gamma x_2$ respectively such that $U\cap V=\emptyset$. Let $f\in C(X)$ be such that $0\le f \le 1$, $f(x_1)=1$ and $\text{Supp}(f)\subset U$. 
Then, $a=(\lambda(e)-\lambda(s))f(\lambda(e)-\lambda(s))=f-f\lambda(s)-s.f\lambda(s)+s.f$. 
So, applying the canonical conditional expectation $\mathbb{E}$ on it, we see that 
$\mathbb{E}(a)=f+s.f$.
Let us now observe that $\mathbb{E}(a)(x_1)=f(x_1)+f(s^{-1}x_1)\ge 1$. 
On the other hand, $\mathbb{E}(a)(x_2)=f(x_2)+f(sx_2)=0$ since $x_2, sx_2\in V$ and $V\cap U=\emptyset$.    
\end{example}
The following example shows that the assumption on $X$ in the above example is necessary as long as the action is non-trivial.  
\begin{example}
\thlabel{ex:necessaryassumption}
Let $X=\{x_1,x_2\}$ be a two point space and $\Gamma=\{e,s\}$ denote $\mathbb{Z}/2\mathbb{Z}$. In this case, $C(X)\rtimes_r\Gamma$ can be identified with $\mathbb{M}_2(\mathbb{C)}$ via the map 
$$u \lambda(e)+v\lambda(s) \mapsto \begin{bmatrix}
    u(x_1) & v(x_1)\\v(x_2) &u(x_2)
\end{bmatrix}. $$
Moreover, under this identification, we see that:
$$C^*_r(\Gamma) \mapsto  \left\{ \left. \begin{bmatrix}
    z & w\\w &z
\end{bmatrix} \right| z,w\in \C\right\}, \qquad
C(X) \mapsto  \left\{ \left. \begin{bmatrix}
    z & 0\\0 &w
\end{bmatrix} \right| z,w\in \C\right\}.
$$
It can now be easily verified that there is no intermediate algebra between $C_r^*(\Gamma)$ and $C(X)\rtimes_r\Gamma$.
\end{example}
We now proceed to deal with all the other non-$C^*$-simple groups. 
We start with the infinite i.c.c groups. 
\subsection{Intermediate algebras for i.c.c. group actions.}
Let us recall that an infinite group $\Gamma$ is i.c.c.. if the conjugacy class of every non-trivial group element is infinite. 
\begin{lemma}
\thlabel{lem:anelementseparating}   Let $\Gamma$ be an i.c.c. group. Let $F\subset\Gamma\setminus\{e\}$ be a finite subset. Then, there exists an element $t\in \Gamma$ such that $tFt^{-1}\cap F=\emptyset$. 
\begin{proof}
Write $F=\{s_i: 1\le i\le n\}$. Let $\Gamma(s_i,s_j)=\{t\in \Gamma: ts_it^{-1}=s_j\}$.  If the claim doesn't hold, $\Gamma=\cup_{i,j}\Gamma(s_i,s_j)$. Whenever $\Gamma(s_i,s_j)$ is non-empty, it is a left coset of $\Gamma(s_i,s_i)$. It is known that, as a consequence, at least one of the subgroups $\Gamma(s_i,s_i)$ is of finite index (see the first few paragraphs of \cite{neumann1954groups}). Consequently, $s_i$ has a finite conjugacy class, contradicting our assumption that
$\Gamma$ is an i.c.c. group.
\end{proof}
\end{lemma}

\begin{theorem}
\thlabel{i.c.c-intermediate}
Let $\Gamma$ be an i.c.c. group and $I\le C_r^*(\Gamma)$, a non-trivial closed two-sided ideal. Let $\mathcal{A}\ne\mathbb{C}$ be a unital $\Gamma$-$C^*$-algebra with a faithful $\Gamma$-invariant state $\varphi$. Then, $\mathcal{B}=C_r^*(\Gamma)+\overline{A_I}$ is not a crossed product in a canonical way. 
\begin{proof}
The fact that $\mathcal{B}$ is an intermediate algebra is a consequence of \thref{algebra}. It also follows from \thref{interm} that $\mathcal{B}\cap\mathcal{A}=\mathbb{C}$ so that $\cB \lneqq \cA \rtimes_r \Gamma$. We show that $\overline{\mathbb{E}\left(\cA_I\right)}=\mathcal{A}$. This will show that $C^*_r(\Gamma) \lneqq \cB$ and complete the proof. 

Let $a\in \mathcal{A}$ and $\eta\in I$ be fixed. Without any loss of generality, assume that $\|a\|=\|\eta\|=1$. Moreover, replacing $\eta$ by $\eta^*\eta$ if required, we can assume that $\tau_0(\eta)\ne 0$. Let $1>\epsilon>0$. We can find a finite subset $F\subset\Gamma\setminus\{e\}$ such that \[\left\|\eta-\sum_{s\in F}c_s\lambda(s)-\tau_0(\eta)\right\|<\frac{\epsilon}{4}.\]
This, in particular, implies that
\[\left\|\sum_{s\in F}c_s\lambda(s)+\tau_0(\eta)\right\|\le\left\|\sum_{s\in F}c_s\lambda(s)+\tau_0(\eta)-\eta\right\|+\|\eta\|<2. \]
Let $t=t(F,\epsilon)\in\Gamma$ (guaranteed by \thref{lem:anelementseparating}) be such that $tFt^{-1}\cap F=\emptyset$.
Then,
\[\left\|\lambda(t)\eta^*\lambda(t^{-1})-\sum_{s\in F}\overline{c_s}\lambda(ts^{-1}t^{-1})-\overline{\tau_0(\eta)}\right\|<\frac{\epsilon}{4}.\]
Let us also observe that
\begin{align*}
&\mathbb{E}\left(\left(\sum_{s\in F}c_s\lambda(s)+\tau_0(\eta)\right)a\left(\sum_{s\in F}\overline{c_s}\lambda(ts^{-1}t^{-1})+\overline{\tau_0(\eta)}\right)\right)\\&=\mathbb{E}\left(\sum_{s,u\in F}c_s\overline{c_u}(s.a)\lambda(stu^{-1}t^{-1})+\sum_{s\in F}c_s\overline{\tau_0(\eta)}(s.a)\lambda(s)+\sum_{u\in F}\tau_0(\eta)\overline{c_u}a\lambda(tu^{-1}t^{-1})\right)\\&+|\tau_0(\eta)|^2\mathbb{E}(a)   \end{align*}
If $stu^{-1}t^{-1}=e$ for some $s,u\in F$, then it would follow that $s=tut^{-1}$ for $s,u\in F$ and this would contradict the choice of $t$. Therefore, we see that
\[\mathbb{E}\left(\left(\sum_{s\in F}c_s\lambda(s)+\tau_0(\eta)\right)a\left(\sum_{s\in F}\overline{c_s}\lambda(ts^{-1}t^{-1})+\overline{\tau_0(\eta)}\right)\right)=|\tau_0(\eta)|^2\mathbb{E}(a)\]
Now, 
\begin{align*}
&\left\|\eta a \left(\lambda(t)\eta^*\lambda(t^{-1})\right)-\left(\sum_{s\in F}c_s\lambda(s)+\tau_0(\eta)\right)a\left(\sum_{s\in F}\overline{c_s}\lambda(ts^{-1}t^{-1})+\overline{\tau_0(\eta)}\right)\right\|\\&\le\left\|\eta a \left(\lambda(t)\eta^*\lambda(t^{-1})\right)-\left(\sum_{s\in F}c_s\lambda(s)+\tau_0(\eta)\right)a\left(\lambda(t)\eta^*\lambda(t^{-1})\right)\right\|\\&+\left\|\left(\sum_{s\in F}c_s\lambda(s)+\tau_0(\eta)\right)a\left(\lambda(t)\eta^*\lambda(t^{-1})\right)-\left(\sum_{s\in F}c_s\lambda(s)+\tau_0(\eta)\right)a\left(\sum_{s\in F}\overline{c_s}\lambda(ts^{-1}t^{-1})+\overline{\tau_0(\eta)}\right)\right\|\\&\le\left\|\eta-\sum_{s\in F}c_s\lambda(s)+\tau_0(\eta)\right\|\|a\|\|\eta\|\\&+\left\|\sum_{s\in F}c_s\lambda(s)+\tau_0(\eta)\right\|\|a\|\left\|\lambda(t)\eta^*\lambda(t^{-1})-\sum_{s\in F}\overline{c_s}\lambda(ts^{-1}t^{-1})-\overline{\tau_0(\eta)}\right\|\\&\le \frac{\epsilon}{2}+\frac{2\epsilon}{4}=\epsilon.    
\end{align*}
Therefore,
\begin{align*}
&\left\|\mathbb{E}\left(\eta a \left(\lambda(t)\eta^*\lambda(t^{-1})\right)\right)-|\tau_0(\eta)|^2\mathbb{E}(a)\right\|\\&=\left\|\mathbb{E}\left(\eta a \left(\lambda(t)\eta^*\lambda(t^{-1})\right)-\left(\sum_{s\in F}c_s\lambda(s)+\tau_0(\eta)\right)a\left(\sum_{s\in F}\overline{c_s}\lambda(ts^{-1}t^{-1})+\overline{\tau_0(\eta)}\right)\right)\right\|\\&\le\left\|\eta a \left(\lambda(t)\eta^*\lambda(t^{-1})\right)-\left(\sum_{s\in F}c_s\lambda(s)+\tau_0(\eta)\right)a\left(\sum_{s\in F}\overline{c_s}\lambda(ts^{-1}t^{-1})+\overline{\tau_0(\eta)}\right)\right\|\\&\le\epsilon \end{align*}
Since $a$ and $\eta$ are fixed in the beginning and  $\epsilon>0$ is arbitrary, we see that $|\tau_0(\eta)|^2\mathbb{E}(a)\in \overline{\mathbb{E}(\mathcal{A}_I)}$. The claim follows.
\end{proof}
\end{theorem}
\subsection{Intermediate algebras for non-i.c.c. group actions.}
Let $\Gamma$ be a non-i.c.c. group. Let $\Gamma_f$ be the union of all finite conjugacy classes of $\Gamma$. $\Gamma_f$ is known as the $FC$-center of the group $\Gamma$. It is well known that $\Gamma_f\triangleleft\Gamma$ is a normal amenable subgroup (see~\cite[Section~X]{de2007simplicity}). Nonetheless, we include proof for the sake of completeness.
  
\begin{lemma}
\thlabel{finiteconjugacyclass}
$\Gamma_f$ is a normal amenable subgroup of $\Gamma$.
\begin{proof}
 Clearly $$\Gamma_f = \{h \in \Gamma \ | \ [\Gamma:C_{\Gamma}(h)] < \infty\}.$$
 Indeed, $C_{\Gamma}(h)$, the centralizer of $h$ in $\Gamma$, is the point stabilizer for the transitive action of $\Gamma$ on the conjugacy class of $h$. This set is closed under conjugation and inverses. It is also closed under multiplication because $C_{\Gamma}(h_1h_2) > C_{\Gamma}(h_1) \cap C_{\Gamma}(h_2)$. Thus, $\Gamma_{f}$ is a normal subgroup. It follows that any finitely generated subgroup $\Delta = \langle h_1,h_2, \ldots h_n \rangle < \Gamma_f$ is virtually Abelian, as it  contains $\Delta \cap \left(\bigcap_{i=1}^{n} C_{\Gamma}(h_i) \right)$ as a finite index Abelian normal subgroup. Thus, $\Gamma$ itself is locally virtually Abelian, particularly it is amenable. This completes the proof. 
\end{proof}
\end{lemma}

\begin{lemma}
\thlabel{amenable}
Let $\Gamma$ be a non-trivial amenable group. Let $$J=\overline{\left\{\sum_s c_s\lambda(s)\in\mathbb{C}[\Gamma]: \sum_s c_s=0\right\}}.$$ Then, $J$ is 
a non-trivial closed ideal of $C_r^*(\Gamma)$
\begin{proof} 
We observe that \[\left\{\sum_{s\in F}c_s\lambda(s): \sum_sc_s=0, 
 F\subset\Gamma, |F|<\infty\right\}\] is contained in the kernel of the trivial representation $\tau: C_r^*(\Gamma)\to\mathbb{C}$. Therefore, the claim follows. 
\end{proof}
\end{lemma}
\begin{theorem}
\thlabel{non-i.c.c}
Let $\Gamma$ be a non-i.c.c. group with $\abs{\Gamma_f} > 2$. Assume that $\mathcal{A}\ne\mathbb{C}$ is a unital $\Gamma$-$C^*$-algebra with a faithful $\Gamma$-invariant state $\varphi$. Then, there exists an intermediate algebra $\mathcal{B}$ with $C_r^*(\Gamma)\subset\mathcal{B}\subset\mathcal{A}\rtimes_r\Gamma$such that $\mathcal{B}$ is not a crossed product in a canonical way.
\begin{proof} 
Let $J \lhd C^*_r(\Gamma_f)$ be the nontrivial closed ideal given by \thref{amenable}. By \thref{biggeridealassociatedtoasmallerone} this can be extended to a nontrivial ideal $I_J \lhd C^*_r(\Gamma)$, and then $\mathcal{B}=C_r^*(\Gamma)+\overline{\mathcal{A}_{I_J}}$ is an intermediate $C^*$ algebra such that $\cB \cap \cA=\mathbb{C}$ by \thref{interm}.

All that remains to be shown is that $\mathbb{E}(\mathcal{B})\ne\mathbb{C}$. We shall show that $\mathbb{E}\left(\overline{\cA_{I_J}}\right)=\mathcal{A}$ and this will complete the proof. 
Let $s\in\Gamma_f$ be a non-identity element. Since $|\Gamma_f|>2$, we can find an element
$t\in\Gamma_f\setminus\{e\}$ such that $st\ne e$. Let $a\in\mathcal{A}$.  Let $\eta=\lambda(s)-\lambda(e)$ and $\tilde{\eta}=\lambda(e)-\lambda(t)$ be two elements in $J\subset I_J$. We observe that $\eta,\tilde{\eta}\in J$, and therefore, $\eta a \tilde{\eta}\in \overline{\mathcal{A}_{I_J}}$ for all $a\in \mathcal{A}$. Now,
\begin{align*}
\eta (-a) \tilde{\eta}&=\left(\lambda(s)-\lambda(e)\right)(-a)\left(\lambda(e)-\lambda(t)\right)
\\&=-\lambda(s)a+\lambda(s)a\lambda(t)+a-a\lambda(t)\\&=
-\lambda(s)a+s.a\lambda(st)+a-a\lambda(t)
\end{align*}
Since $st\ne e$, 
$\mathbb{E}(\eta (-a) \tilde{\eta})=a$. In particular, $\mathbb{E}\left(\overline{\mathcal{A}_J}\right)=\mathcal{A}$.\\
\end{proof}
\end{theorem}

Whenever $\mathcal{A}$ is a commutative $C^*$-algebra $C(X)$ with $|X|>2$, arguing similarly as in Example~\ref{ex:toycase}, we can remove the assumption of $|\Gamma_f|>2$ in the above theorem. In particular, we have the following.     
\begin{cor}
\thlabel{assumptiononX}
 Let $\Gamma$ be a non-i.c.c. group and $X$ a compact Hausdorff $\Gamma$-space with $|X|>2$. Assume that $X$ has a $\Gamma$-invariant probability measure $\nu$ with full support. Then, there exists an intermediate algebra $\mathcal{B}$ with $C_r^*(\Gamma)\subset\mathcal{B}\subset C(X)\rtimes_r\Gamma$ such that $\mathcal{B}$ is not a crossed product in a canonical way.   
\begin{proof}
Let $\Gamma_f$ be the FC-center of the group $\Gamma$. Let $J \lhd C^*_r(\Gamma_f)$ be the closed non trivial ideal and $I_J \lhd C^*_r(\Gamma)$ the corresponding non-trivial ideal (\thref{biggeridealassociatedtoasmallerone}). Also, let $\mathcal{B}=C_r^*(\Gamma)+\overline{A_{I_J}}$ be the intermediate $C^*$ algebra. If $|\Gamma_f|>2$, then the claim follows from \thref{non-i.c.c}. If $|\Gamma_f|=2$, then we can argue similarly as in the proof of Example~\ref{ex:toycase} to conclude that $\mathbb{E}(\cB)\ne\mathbb{C}$.  
\end{proof}
\end{cor}

Let us conclude by noting that in the special case where the algebra $\cA$ is Abelian, we obtain a complete classification in Theorem~\ref{not-crossed}:
\begin{theorem}
Let $\Gamma$ be a non $C^*$-simple group and $X$ a compact Hausdorff $\Gamma$-space with more than two points and admitting a $\Gamma$-invariant measure of full support. Then, there exists an intermediate subalgebra
$$C^*_r(\Gamma) \subsetneq \cB \subsetneq C(X) \rtimes_r \Gamma$$
with the property that $\cB \cap C(X) = \C$.
\end{theorem}

\begin{proof}
Let us assume that $|X|>2$. If $\Gamma_f=\{e\}$, the assertion follows from \thref{i.c.c-intermediate}. If $\Gamma_f\ne \{e\}$, we obtain the claim from \thref{assumptiononX}.
\end{proof}
It is possible that in the case when $|X| = 2$, there might be no nontrivial intermediate subalgebras;
see Example \ref{ex:necessaryassumption}.
\subsection{An intermediate \texorpdfstring{$C^*$-}{}algebra associated to an ideal-II} In this section, for a non-trivial ideal $I\triangleleft\mathcal{A}$, we associate an algebra $\mathcal{B}$ with $\mathcal{A}\subset\mathcal{B}\subset\mathcal{A}\rtimes_r\Gamma$ such that $\mathcal{B}$ is not of the form $\mathcal{A}\rtimes_r\Lambda$ for any subgroup $\Lambda\triangleleft\Gamma$.
\begin{remark}
\label{rem:algassocideal}   Let $\Gamma$ be a discrete group and $\mathcal{A}$ be a unital $\Gamma$-$C^*$-algebra. Let $I$ be a non-trivial closed $\Gamma$-invariant ideal in $\mathcal{A}$. Then, 
\[\mathcal{A}_I=\overline{\text{Span}\left\{\lambda(s) a: s\in \Gamma,~a\in I \right\}}=I\rtimes_r\Gamma\]
is an ideal of $\mathcal{A}\rtimes_r\Gamma$.
\end{remark}
\begin{theorem}
\thlabel{notcomingfromasubgroup}
Let $\Gamma$ be a discrete group. Let $\mathcal{A}$ be a $\Gamma$-$C^*$-algebra such that $\mathcal{A}$ is not $\Gamma$-simple.  Then, there exits an intermediate $C^*$-algebra $\mathcal{A}\subset\mathcal{B}\subset\mathcal{A}\rtimes_r\Gamma$ such that $\mathcal{B}$ is not of the form $\mathcal{A}\rtimes_r\Lambda$ for any subgroup $\Lambda\le\Gamma$.
\begin{proof} Since $\mathcal{A}$ is not $\Gamma$-simple, let $I$ be a non-trivial $\Gamma$-invariant ideal in $\mathcal{A}$. Let $\mathcal{A}_I$ be the associated $C^*$-algebra as constructed in Remark~\ref{rem:algassocideal}. Let $\mathcal{B}=\mathcal{A}+\mathcal{A}_I$. Since $\mathcal{A}_I\triangleleft\mathcal{A}\rtimes_r\Gamma$, it follows from \cite[Theorem~3.1.7]{murphy1990} that $\mathcal{B}$ is a $C^*$-algebra. Consider the possible Fourier coefficients of any element $b\approx\sum_{g\in\Gamma}b_g\lambda(g)\in\mathcal{B}$. Given that $b=a+x$ for some $a\in\mathcal{A}$ and $x\in I\rtimes_r\Gamma$, we see that the coefficients $b_g$ must all lie in $I$ whenever $g\ne e$. In fact, in this case, the set $\{b_g: g\ne e\}$ becomes equal to $I$. Therefore, it is clear that $\mathcal{B}$ cannot be of the form $\mathcal{A}\rtimes_r\Lambda$, for which the set of all possible $b_g$ would be either $\mathcal{A}$ or $\{0\}$, depending on whether $g\in\Lambda$ or $g\not\in\Lambda$, respectively.
\end{proof}
\end{theorem}
\newpage
\bibliographystyle{amsalpha}
\bibliography{irrational_rotation}
\end{document}